\documentclass{amsart} 
\usepackage{lmodern}
\usepackage{microtype} 
\usepackage[english]{babel} 
\usepackage{hyperref} 
\usepackage{mathrsfs}
\usepackage{color}
\usepackage{tikz} 
\usepackage{pgfplots}

\pgfplotsset{compat=1.17} 

\numberwithin{equation}{section} 
\numberwithin{figure}{section}

\theoremstyle{plain} 
\newtheorem{theorem}{Theorem}[section] 
\newtheorem{lemma}[theorem]{Lemma} 
\newtheorem{corollary}[theorem]{Corollary}

\theoremstyle{definition} 
\newtheorem*{definition}{Definition} 

\DeclareMathOperator{\mre}{Re} 
\DeclareMathOperator{\mim}{Im}  
\DeclareMathOperator{\sinc}{sinc} 
\DeclareMathOperator{\dist}{dist}
\DeclareMathOperator{\spa}{span} 
 
\DeclareMathOperator*{\essinf}{ess\,inf}

\begin{document} 
\title[Orthogonal decomposition of composition operators]{Orthogonal decomposition of composition operators on the $H^2$ space of Dirichlet series} 
\date{\today} 

\author{Ole Fredrik Brevig} 
\address{Department of Mathematics, University of Oslo, 0851 Oslo, Norway} 
\email{obrevig@math.uio.no}

\author{Karl-Mikael Perfekt} 
\address{Department of Mathematical Sciences, Norwegian University of Science and Technology (NTNU), NO-7491 Trondheim, Norway} 
\email{karl-mikael.perfekt@ntnu.no}
\begin{abstract}
	Let $\mathscr{H}^2$ denote the Hilbert space of Dirichlet series with square-summable coefficients. We study composition operators $\mathscr{C}_\varphi$ on $\mathscr{H}^2$ which are generated by symbols of the form $\varphi(s) = c_0s + \sum_{n\geq1} c_n n^{-s}$, in the case that $c_0 \geq 1$. If only a subset $\mathbb{P}$ of prime numbers features in the Dirichlet series of $\varphi$, then the operator $\mathscr{C}_\varphi$ admits an associated orthogonal decomposition. Under sparseness assumptions on $\mathbb{P}$ we use this to asymptotically estimate the approximation numbers of $\mathscr{C}_\varphi$. Furthermore, in the case that $\varphi$ is supported on a single prime number, we affirmatively settle the problem of describing the compactness of $\mathscr{C}_\varphi$ in terms of the ordinary Nevanlinna counting function. We give detailed applications of our results to affine symbols and to angle maps. 
\end{abstract}

\subjclass[2020]{Primary 47B33. Secondary 30B50, 30H10.}

\thanks{K.-M. Perfekt was partially supported by grant EP/S029486/1 of the UK Engineering and Physical Sciences Research Council (EPSRC)}

\maketitle

\section{Introduction} Let $\mathscr{H}^2$ be the Hilbert space of Dirichlet series $f(s)=\sum_{n\geq1} b_n n^{-s}$ with square-summable coefficients. For real numbers $\theta$, set $\mathbb{C}_\theta = \{s \in\mathbb{C} \,:\, \mre{s}>\theta\}$, and let $\varphi \colon \mathbb{C}_{1/2} \to \mathbb{C}_{1/2}$ be an analytic function. Gordon and Hedenmalm \cite{GH99} established that the composition operator $\mathscr{C}_\varphi f = f\circ \varphi$ defines a bounded composition operator on $\mathscr{H}^2$ if and only if $\varphi$ belongs to the Gordon--Hedenmalm class $\mathscr{G}$. 
\begin{definition}
	The Gordon--Hedenmalm class $\mathscr{G}$ consists of the analytic functions $\varphi\colon \mathbb{C}_{1/2}\to\mathbb{C}_{1/2}$ of the form
	\[\varphi(s) = c_0 s + \sum_{n=1}^\infty c_n n^{-s} = c_0s + \varphi_0(s),\]
	where $c_0$ is a non-negative integer and the Dirichlet series $\varphi_0$ converges uniformly in $\mathbb{C}_\varepsilon$ for every $\varepsilon>0$ and satisfies the following mapping properties: 
	\begin{enumerate}
		\item[(a)] If $c_0=0$, then $\varphi_0(\mathbb{C}_0)\subseteq \mathbb{C}_{1/2}$. 
		\item[(b)] If $c_0\geq1$, then either $\varphi_0(\mathbb{C}_0)\subseteq \mathbb{C}_0$ or $\varphi_0\equiv i \tau$ for some $\tau \in \mathbb{R}$. 
	\end{enumerate}
	We will use the notation $\mathscr{G}_0$ and $\mathscr{G}_{\geq1}$, respectively, for the subclasses (a) and (b). 
\end{definition}

Let be $T$ be a bounded operator on a Hilbert space. The $n$th approximation number $a_n(T)$ is the distance in the operator norm from $T$ to the operators of rank $<n$. Studying the decay of approximation numbers is relevant for compact operators $T$. Indeed, $T$ is compact if and only if $a_n(T)\to0$ as $n\to\infty$. 

Previously, precise results for the approximation numbers of composition operators on $\mathscr{H}^2$ have primarily been available for symbols $\varphi \in \mathscr{G}_0$, see \cite{BB17, BQS16, QS15}. For case (b) of the Gordon--Hedenmalm class, the following theorem, extracted from the proofs of \cite[Thm.~1.2]{BQS16} and \cite[Thm.~8.1]{BQS16}, gives the best known estimates for general $\varphi \in \mathscr{G}_{\geq1}$. Here, and throughout the paper, we define 
\begin{equation}\label{eq:vartheta} 
	\vartheta = \inf_{s \in \mathbb{C}_0} \mre{\varphi_0(s)} 
\end{equation}
for symbols $\varphi \in \mathscr{G}$. 
\begin{theorem}[Bayart--Queff\'{e}lec--Seip \cite{BQS16}] \label{thm:BQSest} 
	Suppose that $\varphi \in \mathscr{G}_{\geq1}$. Then 
	\begin{equation}\label{eq:BQSest} 
		p_n^{-\mre{c_1}} \leq a_n(\mathscr{C}_\varphi) \leq n^{-\vartheta} 
	\end{equation}
	where $p_n$ denotes the $n$th prime number. 
\end{theorem}
Since the proof of Theorem~\ref{thm:BQSest} is fairly short, we will present it in our preliminary section. Note that the asymptotic estimate $p_n \sim n\log{n}$ as $n\to\infty$ is a direct corollary of the prime number theorem.  

To give an example, suppose that $\varphi(s)=c_0 s + c_1$ for some $c_0\geq1$ and $c_1 \in \overline{\mathbb{C}_0}$. Then $\vartheta = \mre{c_1}$, and if $e_n(s) = n^{-s}$ for $n=1,2,3,\ldots$ denotes the standard basis of $\mathscr{H}^2$, then
\[\mathscr{C}_\varphi e_n = n^{-c_1} e_{n^{c_0}}.\]
Hence $a_n (\mathscr{C}_\varphi) = n^{-\mre{c_1}} = n^{-\vartheta}$ in this case, coinciding with the upper bound of \eqref{eq:BQSest}. Note that for all other symbols, where $\varphi_0(s) \not \equiv c_1$, the maximum principle implies that $\vartheta<\mre{c_1}$.

One of the main goals of the present paper is to improve on the estimates \eqref{eq:BQSest} for certain symbols $\varphi$. Specifically, we shall place restrictions on the prime numbers appearing in the Dirichlet series $\varphi_0$. Let $\mathbb{P}$ denote a set of prime numbers and set $\mathscr{M}(\mathbb{P}) = \{n\in\mathbb{N}\,:\, p|n \,\implies\, p \in \mathbb{P}\}$. We say that a Dirichlet series $f$ is \emph{supported} on $\mathbb{P}$ if
\[f(s) = \sum_{n\in \mathscr{M}(\mathbb{P})} b_n n^{-s}.\]
A set of of prime numbers $\mathbb{P}$ is called \emph{sparse} if $\sum_{p \in \mathbb{P}} p^{-1}<\infty$. Our first main result is the following improvement of the lower bound in Theorem~\ref{thm:BQSest}.
\begin{theorem}\label{thm:lowerimp} 
	Suppose that $\varphi \in \mathscr{G}_{\geq1}$. If $\varphi_0$ is supported on a sparse set of prime numbers, then for every $\varepsilon>0$ there is a positive constant $C=C(\varphi_0,\varepsilon)$ such that
	\[a_n(\mathscr{C}_\varphi) \geq C n^{-\vartheta-\varepsilon}.\]
\end{theorem}

Our proof of Theorem~\ref{thm:lowerimp} relies on an orthogonal decomposition of $\mathscr{C}_\varphi$ that is made available by the assumptions that $c_0 \geq 1$ and that $\varphi_0$ is supported on $\mathbb{P}$, see Lemma~\ref{lem:Csum}. Let $\mathbb{P}^\perp$ denote the set of prime numbers not in $\mathbb{P}$. To apply the orthogonal decomposition effectively, we require that $\mathbb{P}$ is sparse, so that the set $\mathscr{M}(\mathbb{P}^\perp)$ has positive density in $\mathbb{N}$.

We also have a more refined result. We say that a set of prime numbers $\mathbb{P}$ is $\nu$-\emph{sparse} for some $0<\nu\leq 1$ if $\sum_{p \in \mathbb{P}} p^{-\nu}<\infty$. In particular, a set of prime numbers is $1$-\emph{sparse} if and only if it is \emph{sparse}.
\begin{theorem}\label{thm:approxmain} 
	Consider a symbol $\varphi \in \mathscr{G}_{\geq1}$ and suppose that $\varphi_0$ is supported on $\mathbb{P}$. 
	\begin{enumerate}
		\item[(a)] If $\mathbb{P}$ is sparse, then there is a constant $C_1=C_1(\varphi_0)$ such that
		\[a_n(\mathscr{C}_\varphi)\geq C_1\|\mathscr{C}_\varphi e_n\|_{\mathscr{H}^2}.\]
		\item[(b)] If $\mathbb{P}$ is $\nu$-\emph{sparse} for some $0<\nu<1$ and $2\vartheta\geq\nu/(1-\nu)$, then there is a constant $C_2=C_2(\varphi_0, \nu)$ such that
		\[a_n(\mathscr{C}_\varphi)\leq C_2\|\mathscr{C}_\varphi e_n\|_{\mathscr{H}^2}.\]
	\end{enumerate}
\end{theorem}
To exemplify the type of estimates which can be obtained from Theorem~\ref{thm:approxmain}, let $\mathbb{P}$ be a set of prime numbers and consider the affine symbol 
\begin{equation}\label{eq:affine} 
	\varphi(s) = c_0s+c_1+\sum_{p \in \mathbb{P}} c_p p^{-s}. 
\end{equation}
The approximation numbers of composition operators generated by affine symbols $\varphi \in \mathscr{G}_0$ have been investigated by Queff\'{e}lec and Seip \cite[Thm.~1.3]{QS15} and by Muthukumar, Ponnusamy, and Queff\'{e}lec \cite[Thm.~4.1]{MPQ18}. Using Theorem~\ref{thm:approxmain}, we shall obtain the following estimate for the approximation numbers of composition operators generated by affine symbols $\varphi \in\mathscr{G}_{\geq1}$.
\begin{corollary}\label{cor:affine} 
	Suppose that $\varphi$ is an affine symbol \eqref{eq:affine} with $c_0\geq1$, $|\mathbb{P}|=d < \infty$, $c_p\neq0$ and $\vartheta>0$. Then $\varphi\in\mathscr{G}$ and for $n\geq2$,
	\[a_n(\mathscr{C}_\varphi) \asymp n^{-\vartheta} (\log{n})^{-\frac{d}{4}}.\]
\end{corollary}
Note that the case $\vartheta=0$ is omitted from Corollary~\ref{cor:affine}. In this case the estimate from Theorem~1.3 (a) fails to be sharp, since it follows from \cite[Thm.~1]{BB17} that $\mathscr{C}_\varphi$ is not compact, and thus that $a_n(\mathscr{C}_\varphi) \asymp 1$ for $n\geq1$. In Theorem~\ref{thm:affine} we shall also consider some examples of affine symbols supported on infinite but very sparse sets of prime numbers.

In the second part of the paper, we will investigate when the composition operator $\mathscr{C}_\varphi$ is compact on $\mathscr{H}^2$. Suppose that $\varphi \in \mathscr{G}_{\geq1}$, and consider the \emph{Nevanlinna counting function} 
\begin{equation}\label{eq:nevanlinna} 
	N_\varphi(w) = \sum_{s \in \varphi^{-1}(\{w\})} \mre{s}, 
\end{equation}
defined for every $w \in \mathbb{C}_0$. Bayart \cite[Prop.~3]{Bayart03} employed the classical Littlewood inequality for the Nevanlinna counting function in the unit disc to establish the Littlewood--type estimate 
\begin{equation}\label{eq:NLest} 
	N_\varphi(w) \leq \frac{\mre{w}}{c_0}. 
\end{equation}
On account of J.~Shapiro's characterization of the compact composition operators on the Hardy space of the unit disc \cite{Shapiro87}, and the Littlewood--type estimate \eqref{eq:NLest}, it seems plausible that the compactness of $\mathscr{C}_\varphi$ on $\mathscr{H}^2$ is related to the requirement that 
\begin{equation}\label{eq:littleo} 
	\lim_{\mre{w}\to0^+} \frac{N_\varphi(w)}{\mre{w}} = 0. 
\end{equation}
Bayart \cite[Thm.~2]{Bayart03} proved that if $\mim{\varphi_0}$ is bounded and \eqref{eq:littleo} holds, then $\mathscr{C}_\varphi$ is compact on $\mathscr{H}^2$. Conversely, Bailleul \cite[Thm.~6]{Bailleul15} established that if $\varphi_0$ is supported on a finite set of prime numbers, $\varphi$ is finitely valent, and $\mathscr{C}_\varphi$ is compact on $\mathscr{H}^2$, then \eqref{eq:littleo} holds. 

We give a complete description in the case that $\varphi_0$ is supported on a single prime.
\begin{theorem}\label{thm:compact} 
	Suppose that $\varphi \in \mathscr{G}_{\geq1}$ and that $\varphi_0$ is supported on $\mathbb{P}=\{p\}$. Then $\mathscr{C}_\varphi$ is compact on $\mathscr{H}^2$ if and only if
	\[\lim_{\mre{w}\to0^+} \frac{N_\varphi(w)}{\mre{w}} = 0.\]
\end{theorem}

To prove this theorem, we will exploit the fact that such functions $\varphi_0$ are periodic (with period $2\pi i /\log{p}$), in addition to the orthogonal decomposition discussed earlier. Accordingly, we will decompose the Nevanlinna counting function \eqref{eq:nevanlinna} into an infinite number of \emph{restricted} counting functions. To handle these restricted counting functions we will rely on some ideas and techniques from our recent paper \cite{BP21}, where the compactness of $\mathscr{C}_\varphi$ was characterized in the case that $\varphi \in \mathscr{G}_0$. Each restricted counting function comes with a change of variable formula, also known as a Stanton formula, that allows us to express $\|\mathscr{C}_\varphi f\|_{\mathscr{H}^2}$ for Dirichlet series $f$ of a certain form, see Lemma~\ref{lem:CoV}.

To conclude the paper we will provide a detailed study of \emph{angle maps}. For $c_0\geq1$, $\vartheta \geq 0 $ and $0<\alpha<1$, consider the symbol $\varphi_{\alpha, \vartheta}(s) = c_0s + \vartheta + \Phi_\alpha(p^{-s})$, where
\[\Phi_\alpha(p^{-s}) = \left(\frac{1-p^{-s}}{1+p^{-s}}\right)^\alpha.\]
If $\vartheta>0$, then Theorem~\ref{thm:approxmain} immediately implies that $a_n(\mathscr{C}_{\varphi_{\alpha, \vartheta}}) \asymp n^{-\vartheta} (\log{n})^{-\frac{1}{2\alpha}}$ for $n\geq2$, see Corollary~\ref{cor:angle1}. Similarly to the case of affine maps discussed above, Theorem~\ref{thm:approxmain} (a) does not provide the correct lower bound when $\vartheta=0$. In this case we shall instead proceed via the change of variable formula of Lemma~\ref{lem:CoV} and detailed analysis of the restricted counting function.
\begin{theorem}\label{thm:angle2} 
	For a positive integer $c_0$ and a real number $0<\alpha<1$, let $\varphi_\alpha(s) = c_0 s + \Phi_\alpha(p^{-s})$. Then $\varphi_\alpha$ is in $\mathscr{G}_{\geq 1}$ and
	\[a_n(\mathscr{C}_{\varphi_\alpha}) \asymp (\log{n})^{\frac{\alpha-1}{2\alpha}}\]
	for $n\geq2$.
\end{theorem}

In the classical setting of $H^2(\mathbb{D})$, detailed studies of the approximation numbers of composition operators generated by symbols that map into an angle are carried out in \cite{LLQR13} and \cite{QS15two}. Via the transference principle of \cite[Sec.~9]{QS15}, these results also yield estimates for the approximation numbers of composition operators $\mathscr{C}_{\psi_\alpha} \colon \mathscr{H}^2 \to \mathscr{H}^2$ generated by angle maps $\psi_\alpha(s)=1/2+\Phi_\alpha(p^{-s})$.

\subsection*{Organization} In the preliminary Section~\ref{sec:prelim} we give the proof of Theorem~\ref{thm:BQSest}, and discuss the notion of vertical limit functions. In Section~\ref{sec:orth} we analyze the orthogonal decomposition of $\mathscr{C}_\varphi$ and prove Theorem~\ref{thm:lowerimp} and Theorem~\ref{thm:approxmain}. In Section~\ref{sec:affine} we apply Theorem~\ref{thm:approxmain} to affine symbols, and in Section~\ref{sec:schatten} to membership in the Schatten classes. In Section~\ref{sec:restricted} we introduce and study restricted counting functions and their associated Stanton formulas. In Section~\ref{sec:cpctproof} we provide the proof of Theorem~\ref{thm:compact}. In Section~\ref{sec:angle} we study the example of angle maps.

\subsection*{Notation} We will sometimes use the notation $f(x) \ll g(x)$ to indicate that there is a constant $C$ such that $f(x) \leq C g(x)$ for all relevant $x$. The notation $\gg$ indicates the reverse estimate, and $f(x) \asymp g(x)$ means that $f(x) \ll g(x)$ and $g(x) \ll f(x)$.

\subsection*{Acknowledgements} The authors thank the anonymous referee for suggesting an improvement to Theorem~\ref{thm:approxmain}.

\section{Preliminaries} \label{sec:prelim} 
We will have use for two additional characterizations of the approximation numbers of a bounded operator $T$ on a Hilbert space $H$, 
\begin{align}
	a_n(T) &= \sup_{\substack{E \subseteq H \\ \dim(E)=n}} \inf_{\substack{x \in E\\ \|x\|=1}} \|Tx\|, \label{eq:maxmin} \\ a_n(T) &= \inf_{\substack{ E \subseteq H \\ \dim(E)=n-1}} \sup_{\substack{x \in E^\perp \\
	\|x\|=1}} \|Tx\|. 
\label{eq:minmax} \end{align}
See for example \cite[Sec.~II.7]{GK69}. Recall also that approximation numbers satisfy the ideal property 
\begin{equation}\label{eq:ideal} 
	a_n(S_1 T S_2) \leq \|S_1\| a_n(T) \|S_2\| 
\end{equation}
for bounded operators $S_1$, $T$, and $S_2$ on a Hilbert space $H$. 

The following demonstration of Theorem~\ref{thm:BQSest}, adapted from \cite{BQS16}, illustrates the use of \eqref{eq:maxmin} and \eqref{eq:ideal}. In the proof, we also make use of the following result from \cite[p.~329]{GH99}. 
\begin{lemma}\label{lem:norm1} 
	If $\varphi \in \mathscr{G}_{\geq1}$, then $\|\mathscr{C}_\varphi\|=1$. 
\end{lemma}
\begin{proof}
	[Proof of Theorem~\ref{thm:BQSest}] We begin with the upper bound in \eqref{eq:BQSest}. Set $\psi(s)=s+\vartheta$. Note, by the definition \eqref{eq:vartheta} of $\vartheta$, that $\varphi-\vartheta$ is in $\mathscr{G}_{\geq1}$. Since $\mathscr{C}_{\varphi-\vartheta} \mathscr{C}_\psi = \mathscr{C}_\varphi$, the ideal property \eqref{eq:ideal} with $S_1 = I$, $T = \mathscr{C}_{\varphi-\vartheta}$, and $S_2=\mathscr{C}_\psi$, therefore yields
	\[a_n(\mathscr{C}_\varphi) \leq \|\mathscr{C}_{\varphi-\vartheta}\|\,a_n(\mathscr{C}_\psi) = n^{-\vartheta},\]
	where the final equality follows from Lemma~\ref{lem:norm1} and the trivial analysis of $\mathscr{C}_\psi$ presented in the introduction.
	
	For the lower bound in \eqref{eq:BQSest}, we choose $E = \spa( \{e_2,e_3,\ldots,e_{p_n}\})$ as the $n$-dimensional subspace of $\mathscr{H}^2$ in \eqref{eq:maxmin}. To estimate the infimum of $\|\mathscr{C}_\varphi f\|_{\mathscr{H}^2}$, for $f(s)=\sum_{j=1}^n b_j p_j^{-s}$ of unit norm, we consider the auxiliary subspace 
	\[F = \spa(\{e_{2^{c_0}}, e_{3^{c_0}}, \ldots, e_{p_n^{c_0}}\})\]
	and deduce from the fundamental theorem of arithmetic, orthogonality, and the Cauchy--Schwarz inequality that
	\[\|\mathscr{C}_\varphi f\|_{\mathscr{H}^2} \geq \sup_{\substack{g \in F \\ \|g\|_{\mathscr{H}^2}=1}} \big|\langle \mathscr{C}_\varphi f, g \rangle_{\mathscr{H}^2}\big| = \Bigg(\sum_{j=1}^{n} |b_j|^2 p_j^{-2\mre{c_1}}\Bigg)^{\frac{1}{2}}.\]
	Taking the infimum on the right-hand side, over all $f \in E$ of unit norm, we obtain the stated lower bound $a_n(\mathscr{C}_\varphi) \geq p_n^{-\mre{c_1}}$. 
\end{proof}

We will now briefly recall a few facts about vertical limit functions and generalized boundary values. Let $\mathbb{T}^\infty$ denote the countable infinite Cartesian product of the unit circle $\mathbb{T}$ in the complex plane, endowed with its Haar measure $\mu_\infty$. Via prime factorization, we may view any $\chi \in \mathbb{T}^\infty$ as a character,
\[\chi(n) = \chi_1^{\alpha_1} \chi_2^{\alpha_2} \cdots \chi_d^{\alpha_d} \qquad \text{for} \qquad n = \prod_{j=1}^d p_j^{\alpha_j}.\]
For a Dirichlet series $f(s)=\sum_{n\geq1}b_n n^{-s}$ and a character $\chi \in \mathbb{T}^\infty$, consider the vertical limit function
\[f_\chi(s) = \sum_{n=1}^\infty b_n \chi(n) n^{-s}.\]
If $f$ converges uniformly in $\overline{\mathbb{C}_\theta}$ for some $\theta \in \mathbb{R}$, then $\{f_\chi\}_{\chi\in\mathbb{T}^\infty}$ consists precisely of the functions which can be obtained as uniform limits in $\overline{\mathbb{C}_\theta}$ of vertical translates $f(\cdot +i\tau_k)$, where $(\tau_k)_{k\geq1}$ is a sequence of real numbers. Despite the fact that a function $f \in \mathscr{H}^2$ need only converge in $\mathbb{C}_{1/2}$, the Dirichlet series $f_\chi$ actually converges in $\mathbb{C}_0$ for almost every $\chi \in \mathbb{T}^\infty$ (see e.g.~\cite[Thm.~4.1]{HLS97}). Moreover, the generalized boundary value
\[f^\ast(\chi) = \lim_{\sigma\to0^+} f_\chi(\sigma)\]
exists for almost every $\chi \in \mathbb{T}^\infty$, and 
\begin{equation}\label{eq:H2astnorm} 
	\|f\|_{\mathscr{H}^2} = \|f^\ast\|_{L^2(\mathbb{T}^\infty)}. 
\end{equation}
The following result can be extracted from \cite[Sec.~2]{BP20}. 
\begin{lemma}\label{lem:vlf} 
	Suppose that $\varphi\colon \mathbb{C}_0\to\mathbb{C}_0$ is a Dirichlet series which converges uniformly in $\overline{\mathbb{C}_\varepsilon}$ for every $\varepsilon>0$. Then 
	\begin{enumerate}
		\item[(i)] $\varphi_\chi(\mathbb{C}_0) = \varphi(\mathbb{C}_0)$ for every $\chi \in \mathbb{T}^\infty$, and 
		\item[(ii)] $\varphi^\ast(\chi)$ exists for almost every $\chi \in \mathbb{T}^\infty$. 
	\end{enumerate}
\end{lemma}
In particular, we deduce from Lemma~\ref{lem:vlf} that if $\varphi(s) = c_0s+\varphi_0(s)$ is in $\mathscr{G}$, then the expression \eqref{eq:vartheta} for $\vartheta$ has the reformulation 
\begin{equation}\label{eq:varthetaast} 
	\vartheta = \essinf_{\chi \in \mathbb{T}^\infty} \mre{\varphi_0^\ast(\chi)}. 
\end{equation}
Following \cite{GH99}, we extend the notion of vertical limit functions to symbols $\varphi \in \mathscr{G}$ by defining
\[\varphi_\chi(s) = c_0s + (\varphi_0)_\chi(s).\]
The interaction between the composition operator $\mathscr{C}_\varphi$ and vertical limits is given in \cite[Prop.~4.3]{GH99}: 
\begin{equation}\label{eq:compformula} 
	(\mathscr{C}_\varphi f)_\chi = \mathscr{C}_{\varphi_\chi} f_{\chi^{c_0}}, 
\end{equation}
where $f \in \mathscr{H}^2$, $\chi \in \mathbb{T}^\infty$, and $\chi^{c_0}(n) = \chi(n)^{c_0} =\chi(n^{c_0})$. Combining Lemma~\ref{lem:vlf}~(ii), \eqref{eq:H2astnorm}, and \eqref{eq:compformula} yields the following result. 
\begin{lemma}\label{lem:UE} 
	If $\varphi \in \mathscr{G}$ and $\chi \in \mathbb{T}^\infty$, then $\mathscr{C}_\varphi$ and $\mathscr{C}_{\varphi_\chi}$ are unitarily equivalent. 
\end{lemma}

\section{Orthogonal decomposition and approximation numbers} \label{sec:orth} 
We now fix a subset $\mathbb{P}$ of the full set of prime numbers. For each $j \in \mathscr{M}(\mathbb{P}^\perp)$, we let $\mathscr{H}^2_j$ denote the subspace of $\mathscr{H}^2$ comprised of Dirichlet series of the form $e_j f$, where $f$ is supported on $\mathbb{P}$. Since $\mathscr{H}^2_{j_1} \perp \mathscr{H}^2_{j_2}$ if $j_1 \neq j_2$, we have the orthogonal decomposition 
\begin{equation}\label{eq:H2sum} 
	\mathscr{H}^2 = \bigoplus_{j \in \mathscr{M}(\mathbb{P}^\perp)} \mathscr{H}^2_j. 
\end{equation}
The following simple observation is the starting point of the present paper. 
\begin{lemma}\label{lem:Csum} 
	Let $\varphi \in \mathscr{G}_{\geq1}$ and suppose that $\varphi_0$ is supported on $\mathbb{P}$. For every $j \in \mathscr{M}(\mathbb{P}^\perp)$, let $\mathscr{C}_{\varphi,j}$ denote the operator obtained by restricting $\mathscr{C}_\varphi$ to $\mathscr{H}^2_j$. Then
	\[\mathscr{C}_\varphi = \bigoplus_{j \in \mathscr{M}(\mathbb{P}^\perp)} \mathscr{C}_{\varphi,j}.\]
\end{lemma}
\begin{proof}
	In view of \eqref{eq:H2sum}, it is sufficient to prove that $\mathscr{C}_\varphi$ maps $\mathscr{H}^2_j$ to $\mathscr{H}^2_{j^{c_0}}$, since the map $j \mapsto j^{c_0}$ is injective on $\mathscr{M}(\mathbb{P}^\perp)$. Consider the action of $\mathscr{C}_\varphi$ on $e_n$, where $n=jk$ for $j \in \mathscr{M}(\mathbb{P}^\perp)$ and $k \in \mathscr{M}(\mathbb{P})$:
	\[\mathscr{C}_\varphi e_n (s) = j^{-c_0 s} k^{-c_0 s} n^{-\varphi_0(s)}.\]
	We see that $\mathscr{C}_\varphi e_n \in \mathscr{H}^2_{j^{c_0}}$, as a consequence of the assumption that $\varphi_0$ is supported on $\mathbb{P}$. 
\end{proof}

In view of Lemma~\ref{lem:Csum} there is for every $n\geq1$ some $m\geq1$ and $j \in\mathscr{M}(\mathbb{P}^\perp)$ such that $a_n(\mathscr{C}_\varphi) = a_m(\mathscr{C}_{\varphi,j})$. We first apply this to obtain a lower bound for the approximation numbers of $\mathscr{C}_\varphi$ which will immediately imply our first main result. 
\begin{lemma}\label{lem:sparse} 
	Let $\varphi \in \mathscr{G}_{\geq1}$ and suppose that $\varphi_0$ is supported on a sparse set of prime numbers $\mathbb{P}$. There is then a positive integer $m=m(\mathbb{P})$ such that
	\[a_n(\mathscr{C}_\varphi) \geq \|\mathscr{C}_\varphi e_{mn}\|_{\mathscr{H}^2}.\]
\end{lemma}
\begin{proof}
	By definition, any $f_j \in \mathscr{H}^2_j$ can be written $f_j = e_j f$ for a function $f$ supported on $\mathbb{P}$, and $\|f_j\|_{\mathscr{H}^2}=\|f\|_{\mathscr{H}^2}$. By Lemma~\ref{lem:vlf} (ii) and the composition rule \eqref{eq:compformula}, we have that
	\[(\mathscr{C}_\varphi f_j)^\ast(\chi) = \chi^{c_0}(j) j^{-\varphi_0^\ast(\chi)} f_{\chi^{c_0}}(\varphi_0^\ast(\chi))\]
	for almost every $\chi \in \mathbb{T}^\infty$. This formula is at first valid for polynomials $f$, but by a density argument it continues to hold for general $f$ supported on $\mathbb{P}$, if we interpret $f_{\chi^{c_0}}(\varphi_0^\ast(\chi))$ as a generalized boundary value when needed. By \eqref{eq:H2astnorm} we therefore have that
	\[\|\mathscr{C}_\varphi f_j\|_{\mathscr{H}^2}^2 = \int_{\mathbb{T}^\infty} j^{-2\mre{\varphi_0^\ast(\chi)}} \left|f_{\chi^{c_0}}(\varphi_0^\ast(\chi))\right|^2\,d\mu_\infty(\chi).\]
	In particular, $j \mapsto \|\mathscr{C}_{\varphi,j}\|=a_1(\mathscr{C}_{\varphi,j})$ is decreasing for $j\in\mathscr{M}(\mathbb{P}^\perp)$. Letting $(j_n)_{n\geq1}$ denote the increasing sequence of integers in $\mathscr{M}(\mathbb{P}^\perp)$, we conclude that
	\[a_n(\mathscr{C}_\varphi) \geq a_1(\mathscr{C}_{\varphi,j_n})=\|\mathscr{C}_{\varphi,j_n}\|\geq \|\mathscr{C}_\varphi e_{j_n}\|_{\mathscr{H}^2},\]
	since $e_{j_n}\in\mathscr{H}^2_{j_n}$ and $\|e_{j_n}\|_{\mathscr{H}^2}=1$. The hypothesis that $\mathbb{P}$ is sparse means that
	\[\lim_{N\to\infty} \frac{1}{N} \operatorname{card}\left\{ j \in \mathscr{M}(\mathbb{P}^\perp) \,:\, j \leq N\right\} = \prod_{p \in \mathbb{P}} \left(1-\frac{1}{p}\right) = C(\mathbb{P}) \neq 0,\]
	and thus that there is a positive integer $m$ such that $j_n \leq mn$ for every $n\geq 1$. Therefore
	\[\|\mathscr{C}_\varphi e_{j_n}\|_{\mathscr{H}^2}^2 \geq \int_{\mathbb{T}^\infty} (mn)^{-2\mre{\varphi_0^\ast(\chi)}} \,d\mu_\infty(\chi) = \|\mathscr{C}_\varphi e_{mn}\|_{\mathscr{H}^2}^2. \qedhere \]
\end{proof}
\begin{proof}
	[Proof of Theorem~\ref{thm:lowerimp}] Since $\varphi_0$ is supported on a sparse set of prime numbers, Lemma~\ref{lem:sparse} yields that
	\[a_n(\mathscr{C}_\varphi) \geq \|\mathscr{C}_\varphi e_{mn}\|_{\mathscr{H}^2}\]
	for some positive integer $m$. Set $X_\varepsilon=\left\{\chi \in \mathbb{T}^\infty\,:\,\vartheta \leq \mre{\varphi_0^\ast(\chi)} \leq \vartheta+\varepsilon\right\}$. Then $\mu_\infty(X_\varepsilon)>0$, referring to \eqref{eq:varthetaast}, and accordingly
	\[\|\mathscr{C}_\varphi e_{mn}\|_{\mathscr{H}^2}^2 = \int_{\mathbb{T}^\infty} (mn)^{-2\mre{\varphi_0^\ast(\chi)}}\,d\mu_\infty(\chi)\geq \mu_\infty(X_\varepsilon) (mn)^{-2(\vartheta+\varepsilon)}.\]
	This gives the stated estimate with $C=\sqrt{\mu_\infty(X_\varepsilon)} m^{-\vartheta-\varepsilon}$. 
\end{proof}

We now turn toward proving Theorem~\ref{thm:approxmain} (a).
\begin{lemma}\label{lem:doubling} 
	Suppose that $\varphi \in \mathscr{G}_{\geq1}$ and let $m$ be a positive integer. There is a constant $C=C(\varphi_0,m)>0$ such that
	\[\|\mathscr{C}_\varphi e_n\|_{\mathscr{H}^2} \leq C \|\mathscr{C}_\varphi e_{mn}\|_{\mathscr{H}^2}\]
	for every integer $n\geq1$.
\end{lemma}

\begin{proof}
	As before, we compute the norms on $\mathbb{T}^\infty$, so that
	\[\|\mathscr{C}_\varphi e_n\|_{\mathscr{H}^2}^2 = \int_{\mathbb{T}^\infty} n^{-2\mre{\varphi_0^\ast(\chi)}}\,d\mu_\infty(\chi).\]
	For any $\varepsilon>0$, consider the set $X_\varepsilon=\left\{\chi \in \mathbb{T}^\infty\,:\,\vartheta \leq \mre{\varphi_0^\ast(\chi)} \leq \vartheta+\varepsilon\right\}$. As in the proof of Theorem~\ref{thm:lowerimp}, we know that $\mu_\infty(X_\varepsilon)>0$. Since $x \mapsto n^{-x}$ is non-increasing for $x>0$, it follows, by interpreting each side of the inequality as an average, that
	\[\|\mathscr{C}_\varphi e_n\|_{\mathscr{H}^2}^2 \leq \frac{1}{\mu_\infty(X_\varepsilon)} \int_{X_\varepsilon} n^{-2\mre{\varphi_0^\ast(\chi)}}\,d\mu_\infty(\chi).\]
 By the definition of $X_\varepsilon$, we find that
	\[\int_{X_\varepsilon} n^{-2\mre{\varphi_0^\ast(\chi)}}\,d\mu_\infty(\chi) \leq m^{2(\vartheta+\varepsilon)} \int_{X_\varepsilon} (mn)^{-2\mre{\varphi_0^\ast(\chi)}}\,d\mu_\infty(\chi).\]
	Extending the final integral from $X_\varepsilon$ to $\mathbb{T}^\infty$, we conclude that
	\[\|\mathscr{C}_\varphi e_n\|_{\mathscr{H}^2}^2 \leq \frac{m^{2(\vartheta+\varepsilon)}}{\mu_\infty(X_\varepsilon)} \|\mathscr{C}_\varphi e_{mn}\|_{\mathscr{H}^2}^2. \qedhere\]
\end{proof}

\begin{proof}
	[Proof of Theorem~\ref{thm:approxmain} (a)] Combining Lemma~\ref{lem:sparse} and Lemma~\ref{lem:doubling} yields that
	\[a_n(\mathscr{C}_\varphi) \geq \|\mathscr{C}_\varphi e_{mn}\|_{\mathscr{H}^2} \geq C^{-1} \|\mathscr{C}_\varphi e_n\|_{\mathscr{H}^2},\]
	where $m$ is as in Lemma~\ref{lem:sparse} and $C$ is from Lemma~\ref{lem:doubling}. 
\end{proof}

The remainder of this section is devoted to the proof of Theorem~\ref{thm:approxmain} (b). For notational reasons, we introduce the partial zeta function
\[\zeta_{\mathbb{P}}(s) = \prod_{p \in \mathbb{P}} \frac{1}{1-p^{-s}}.\]
It is clear that if $\mathbb{P}$ is $\nu$-sparse for some $0<\nu\leq1$, then $\zeta_{\mathbb{P}}(\nu)<\infty$.
\begin{lemma}\label{lem:abelsum} 
	Suppose that $\mathbb{P}$ is a set of $\nu$-sparse prime numbers for some $0 < \nu \leq 1$. Then
	\[\sum_{\substack{k \in \mathscr{M}(\mathbb{P}) \\ k \geq K}} k^{-2\sigma} \leq \zeta_{\mathbb{P}}(\nu) K^{\nu-2\sigma}\]
	for every $K \in \mathscr{M}(\mathbb{P})$ and every $2\sigma \geq \nu$. 
\end{lemma}
\begin{proof}
	We estimate
	\[\sum_{\substack{k \in \mathscr{M}(\mathbb{P}) \\ k \geq K}} k^{-2\sigma} \leq K^{\nu-2\sigma} \sum_{\substack{k \in \mathscr{M}(\mathbb{P}) \\ k \geq K}} k^{-\nu } \leq K^{\nu-2\sigma} \sum_{k \in \mathscr{M}(\mathbb{P})} k^{-\nu} = K^{\nu-2\sigma} \zeta_{\mathbb{P}}(\nu). \qedhere\]
\end{proof}
\begin{lemma}\label{lem:upperest} 
	Fix $\varphi \in \mathscr{G}_{\geq1}$ and suppose that $\varphi_0$ is supported on a $\nu$-sparse set of prime numbers $\mathbb{P}$ for some $0 < \nu \leq 1$. If $2\vartheta \geq \nu$, then
	\[a_m(\mathscr{C}_{\varphi,j}) \leq \sqrt{\zeta_{\mathbb{P}}(\nu)} k_m^{\nu/2} \|\mathscr{C}_\varphi e_{j k_m} \|_{\mathscr{H}^2},\]
	where $(k_m)_{m\geq1}$ are the integers of $\mathscr{M}(\mathbb{P})$ in increasing order and $j \in \mathscr{M}(\mathbb{P}^\perp)$. 
\end{lemma}
\begin{proof}
	We apply the min-max principle \eqref{eq:minmax}, choosing $E \subseteq \mathscr{H}^2_j$ as
	\[E = \spa \left(\left\{e_{jk_1}, e_{j k_2}, \ldots, e_{j k_{m-1}}\right\}\right).\]
	This gives us that
	\[a_m(\mathscr{C}_{\varphi,j}) \leq \sup_{\substack{f \in E^\perp \\ \|f\|_{\mathscr{H}^2}=1}} \|\mathscr{C}_\varphi f\|_{\mathscr{H}^2}.\]
	Accordingly, suppose that $f \in E^\perp$ with $\|f\|_{\mathscr{H}^2}=1$. If $\mre{s}\geq \vartheta$, the Cauchy--Schwarz inequality and Lemma~\ref{lem:abelsum} imply that $f(s)$ converges absolutely, and that
	\[|f(s)|^2 \leq \sum_{\substack{k \in \mathscr{M}(\mathbb{P}) \\ k \geq k_m}} (jk)^{-2\mre{s}} = j^{-2\mre{s}}\sum_{\substack{k \in \mathscr{M}(\mathbb{P}) \\ k \geq k_m}} k^{-2\mre{s}} \leq \zeta_{\mathbb{P}}(\nu) k_m^\nu (j k_m)^{-2\mre{s}}.\]
	Of course the same estimate also holds if $f$ is replaced by $f_{\chi^{c_0}}$ for any $\chi \in \mathbb{T}^\infty$. Since $s=\mre{\varphi_0^\ast(\chi)} \geq \vartheta$ for almost every $\chi$, we may therefore apply this estimate in conjunction with Lemma~\ref{lem:vlf} (ii), \eqref{eq:H2astnorm} and \eqref{eq:compformula} to see that 
	\begin{align*}
		\|\mathscr{C}_\varphi f\|_{\mathscr{H}^2}^2 &= \int_{\mathbb{T}^\infty} |f_{\chi^{c_0}}(\varphi_0^\ast(\chi))|^2 \,d\mu_\infty(\chi) \\
		&\leq \int_{\mathbb{T}^\infty} \zeta_{\mathbb{P}}(\nu) k_m^\nu (j k_m)^{-2\mre{\varphi_0^\ast(\chi)}}\,d\mu_\infty(\chi)= \zeta_{\mathbb{P}}(\nu)k_m^\nu \|\mathscr{C}_\varphi e_{j k_m}\|_{\mathscr{H}^2}^2. 
	\end{align*}
	Together with the min-max principle, this gives the claimed estimate. 
\end{proof}
\begin{proof}
	[Proof of Theorem~\ref{thm:approxmain} (b)] The function $\Phi \colon [1,\infty) \to (0,1]$ defined by
	\[\Phi(x) = \left(\int_{\mathbb{T}^\infty} x^{-2\mre{\varphi_0^\ast(\chi)}}\,d\mu_\infty(\chi)\right)^\frac{1}{2}\]
	is strictly decreasing, onto (by the assumption $\vartheta>0$), continuous and enjoys the estimate $\Phi(xy) \leq y^{-\vartheta} \Phi(x)$ for every $x,y\geq1$. Hence $\Phi$ has an inverse function $\Phi^{-1}\colon (0,1]\to [1,\infty)$ satisfying the same properties and enjoying the estimate 
	\begin{equation}\label{eq:invest} 
		\Phi^{-1}(xy) \leq y^{-1/\vartheta} \Phi^{-1}(x) 
	\end{equation}
	for every $0<x,y \leq 1$. Fix some $0<x\leq1$. The orthogonal decomposition of Lemma~\ref{lem:Csum} allows us to rewrite
	\[\big|\big\{n \in \mathbb{N} \,:\, a_n(\mathscr{C}_\varphi) \geq \sqrt{\zeta_{\mathbb{P}}(\nu)} \, x \big\}\big| = \big|\big\{(j,m) \in \mathscr{M}(\mathbb{P}^\perp) \times \mathbb{N} \,:\, a_m(\mathscr{C}_{\varphi_j}) \geq \sqrt{\zeta_{\mathbb{P}}(\nu)} x \big\}\big|.\]
	We now apply Lemma~\ref{lem:upperest} to bound the right-hand side from above. Note that the hypotheses of Lemma~\ref{lem:upperest} certainly hold, since we are working under the stronger assumptions that $0 < \nu < 1$ and $2\vartheta\geq \nu/(1-\nu)$. We obtain that 
	\begin{align*}
		\big|\big\{n \in \mathbb{N} \,:\, a_n(\mathscr{C}_\varphi) \geq \sqrt{\zeta_{\mathbb{P}}(\nu)} \, x \big\}\big| &\leq \big|\big\{(j,m) \in \mathscr{M}(\mathbb{P}^\perp) \times \mathbb{N} \,:\, k_m^{\nu/2} \Phi(j k_m ) \geq x \big\}\big| \\
		&= \big|\big\{(j,m) \in \mathscr{M}(\mathbb{P}^\perp) \times \mathbb{N} \,:\, j \leq \Phi^{-1}(x k_m^{-\nu/2})/k_m \big\}\big|. 
	\end{align*}
	Counting for each $m$ the number of positive integers $j$ (not only those in $\mathscr{M}(\mathbb{P}^\perp)$) which satisfy the inequality $j \leq \Phi^{-1}(x k_m^{-\nu/2})/k_m$, we therefore have the upper bound
	\[\big|\big\{n \in \mathbb{N} \,:\, a_n(\mathscr{C}_\varphi) \geq \sqrt{\zeta_{\mathbb{P}}(\nu)} \, x \big\}\big| \leq \sum_{m=1}^\infty \frac{\Phi^{-1}\big(x k_m^{-\nu/2}\big)}{k_m} \leq \Phi^{-1}(x) \zeta_{\mathbb{P}}(1-\nu/(2\vartheta)),\]
	where the second inequality comes from \eqref{eq:invest} applied with $y = k_m^{-\nu/2} \leq 1$. Since $2\vartheta\geq \nu/(1-\nu)$, we conclude that the estimate 
	\begin{equation}\label{eq:letsgo} 
		\big|\big\{n \in \mathbb{N} \,:\, a_n(\mathscr{C}_\varphi) \geq \sqrt{\zeta_{\mathbb{P}}(\nu)} \, x \big\}\big| \leq \zeta_{\mathbb{P}}(\nu) \,\Phi^{-1}(x) 
	\end{equation}
	holds for every $0<x\leq1$. 
	
	Since $\vartheta>0$, there is a smallest positive integer $N$ such that $N^{2\vartheta}\geq \zeta_{\mathbb{P}}(\nu)$. By the upper bound in Theorem~\ref{thm:BQSest} it follows that $a_n(\mathscr{C}_\varphi) \leq \sqrt{\zeta_{\mathbb{P}}(\nu)}$ for every $n\geq N$. Applying \eqref{eq:letsgo} with $x =a_n(\mathscr{C}_\varphi)/\sqrt{\zeta_{\mathbb{P}}(\nu)} \leq 1$ immediately gives us that 
	\begin{equation}\label{eq:nbe} 
		a_n(\mathscr{C}_\varphi) \leq \sqrt{\zeta_{\mathbb{P}}(\nu)} \Phi\left(\frac{n}{\zeta_{\mathbb{P}}(\nu)}\right) 
	\end{equation}
	for every $n \geq N$. Following the proof of Lemma~\ref{lem:doubling} verbatim with $\varepsilon=\vartheta$ yields that 
	\begin{equation}\label{eq:dagain} 
		\Phi\left(\frac{n}{\zeta_{\mathbb{P}}(\nu)}\right) \leq \frac{\big(\zeta_{\mathbb{P}}(\nu)\big)^{2\vartheta}}{\sqrt{\mu_\infty(X)}} \Phi(n), 
	\end{equation}
	where the set $X = \left\{\chi \in \mathbb{T}^\infty\,:\,\vartheta \leq \mre{\varphi_0^\ast(\chi)} \leq 2\vartheta\right\}$ satisfies $\mu_\infty(X)>0$. Combining \eqref{eq:nbe} and \eqref{eq:dagain}, we conclude that
	\[a_n(\mathscr{C}_\varphi) \leq \frac{\big(\zeta_{\mathbb{P}}(\nu)\big)^{1/2+2\vartheta}}{\sqrt{\mu_\infty(X)}} \|\mathscr{C}_\varphi e_n\|_{\mathscr{H}^2}\]
	for every $n \geq N$, which completes the proof. 
\end{proof}

\section{Composition operators generated by affine symbols} \label{sec:affine} 
To exemplify Theorem~\ref{thm:approxmain} we consider affine symbols, which we recall from \eqref{eq:affine} to have the form
\[\varphi(s) = c_0 s + c_1 + \sum_{p\in\mathbb{P}} c_p p^{-s}.\]
At first, we assume that $\varphi$ is supported by a set of $|\mathbb{P}|=d < \infty$ prime numbers. In particular, $c_p\neq0$ for every $p \in \mathbb{P}$. Note from \eqref{eq:varthetaast} that
\[\vartheta = \mre{c_1}-\sum_{p\in\mathbb{P}} |c_p|.\]

Before proving Corollary~\ref{cor:affine}, let us quickly recall the known results about $a_n(\mathscr{C}_\varphi)$ in this setting. We begin with the case $c_0=0$, in which case we must require that $\vartheta\geq 1/2$ in order for $\mathscr{C}_\varphi$ to be bounded. Queff\'{e}lec and Seip \cite[Thm.~1.3]{QS15} have established that if $\vartheta=1/2$, then
\[\left(\frac{1}{n}\right)^{(d-1)/2} \ll a_n(\mathscr{C}_\varphi) \ll \left(\frac{\log{n}}{n}\right)^{(d-1)/2}.\]
If $\vartheta>1/2$, then by \cite[Thm.~4.1]{MPQ18} we have that
\[a_n(\mathscr{C}_\varphi) \ll \left(\frac{\mre{c_1}-\vartheta}{\mre{c_1}-1/2}\right)^n,\]
where the implied constant depends on $\mre{c_1}$ and $\vartheta>1/2$, but \emph{not} on $d$. Actually, the estimate is stated and proved only for $d=1$ in \cite{MPQ18}. However, it can be extended to general $d\geq1$ by applying the max-min principle \eqref{eq:maxmin} and the subordination principle for affine symbols from \cite[Thm.~5]{BP20}. 

Suppose instead that $c_0\geq1$. If $\vartheta>0$, then the best previously known estimates were from Theorem~\ref{thm:BQSest}. As mentioned in the introduction, if $\vartheta=0$ for an affine symbol $\varphi$, then $a_n(\mathscr{C}_\varphi)\asymp 1$ for $n\geq1$, and so this case is not of interest.

To prove Corollary~\ref{cor:affine}, we require the following version of Hankel's asymptotic estimate for the modified Bessel function of the second kind with parameter $0$. It will be convenient for us to have explicit constants; we have made no attempt to optimize these.
\begin{lemma}\label{lem:hankel} 
	If $x \geq \frac{1}{8}$, then
	\[\frac{1}{\pi \sqrt{2e}}\frac{1}{\sqrt{x}} \leq \int_{-\pi}^\pi e^{- 4x \sin^2(\theta/2)}\,\frac{d\theta}{2\pi} \leq \frac{\sqrt{\pi}}{4} \frac{1}{\sqrt{x}}.\]
\end{lemma}
\begin{proof}
	For the upper bound, we use that $|\sin(\theta/2)| \geq |\theta/\pi|$ for $-\pi \leq \theta \leq \pi$ to conclude that
	\[\int_{-\pi}^\pi e^{- 4x \sin^2(\theta/2)}\,\frac{d\theta}{2\pi} \leq \int_{-\infty}^\infty e^{- \frac{4x}{\pi^2}\,\theta^2}\,\frac{d\theta}{2\pi} = \frac{\sqrt{\pi}}{4} \frac{1}{\sqrt{x}}.\]
	For the lower bound, we suppose that $0<\varepsilon \leq 2\sqrt{x}$. Then
	\[\int_{-\pi}^\pi e^{- 4x \sin^2(\theta/2)}\,\frac{d\theta}{2\pi} \geq \frac{e^{-\varepsilon^2}}{2\pi}\,\left|\left\{-\pi \leq \theta \leq \pi \,:\, |\sin(\theta/2)| \leq \frac{\varepsilon}{2\sqrt{x}} \right\}\right| \geq \frac{\varepsilon e^{-\varepsilon^2}}{\pi} \frac{1}{\sqrt{x}}.\]
	where we used that $|\sin(\theta/2)|\leq |\theta/2|$ for the final inequality. The stated lower bound is obtained by choosing $\varepsilon = 1/\sqrt{2}$, which is permissible by the assumption that $x \geq 1/8$. 
\end{proof}
\begin{proof}
	[Proof of Corollary~\ref{cor:affine}] In view of Lemma~\ref{lem:UE} we may without loss of generality replace $\varphi$ by $\varphi_\chi$ for any $\chi \in \mathbb{T}^\infty$ This allows us to assume that $\varphi_0$ is of the form
	\[\varphi_0(s)=\vartheta + i \tau + \sum_{p\in\mathbb{P}} \gamma_p (1-p^{-s}),\]
	where $\tau\in\mathbb{R}$ and $\gamma_p>0$ for $p \in \mathbb{P}$. By Theorem~\ref{thm:approxmain} (a) and (b), we need to estimate
	\[\|\mathscr{C}_\varphi e_n\|_{\mathscr{H}^2}^2 = \int_{\mathbb{T}^\infty} n^{-2\mre{\varphi_0^\ast(\chi)}}\,d\mu_\infty(\chi) = n^{-2\vartheta} \prod_{p\in\mathbb{P}} \int_{-\pi}^\pi n^{-2\gamma_p(1-\cos{\theta_p})}\,\frac{d\theta_p}{2\pi}\]
	as $n\to\infty$. Suppose that $n$ is large enough that $\gamma_p \log{n} \geq \frac{1}{8}$ for every $p \in \mathbb{P}$. Then, by applying Lemma~\ref{lem:hankel} with $x = \gamma_p \log{n}$,
	\[\prod_{p\in\mathbb{P}} \int_{-\pi}^\pi n^{-2\gamma_p (1-\cos{\theta_p})}\,\frac{d\theta_p}{2\pi} = \prod_{p\in\mathbb{P}}\int_{-\pi}^\pi n^{-4\gamma_p \sin^2(\theta_p/2)}\,\frac{d\theta_p}{2\pi} \asymp (\log{n})^{-\frac{d}{2}}. \qedhere\]
\end{proof}

We finish this section by discussing a class of affine symbols with $|\mathbb{P}| = \infty$. For any affine symbol with absolutely convergent coefficients, the image $\varphi_0^\ast(\mathbb{T}^\infty)$ is an annulus (see e.g.~\cite[Sec.~XI.5]{Titchmarsh86}). Hence $\varphi_0^\ast(\mathbb{T}^\infty)$ touches the line $\mre w = \vartheta$ tangentially. However, the examples of this section show that the interaction between different prime numbers is essential in determining the behavior of the approximation numbers $a_n(\mathscr{C}_\varphi)$. When $c_0=0$, symbols with $|\mathbb{P}| = \infty$ have previously been considered in \cite[Thm.~1.3]{QS15}.
\begin{theorem}\label{thm:affine} 
	Let $\mathbb{P}=(p_j)_{j\geq1}$ be a set of prime numbers which is $\nu$-sparse for every $\nu>0$. For fixed $c_0\geq1$, $\vartheta>0$, and $\beta>1$, define
	\[\varphi(s) = c_0 s + \vartheta + \sum_{j=1}^\infty \frac{1-p_j^{-s}}{j^\beta}.\]
	Then there are positive constants $C_1=C_1(\beta)$ and $C_2=C_2(\beta)$ such that
	\[n^{-\vartheta} e^{-C_1 (\log{n})^{1/\beta}} \ll a_n(\mathscr{C}_\varphi) \ll n^{-\vartheta} e^{-C_2 (\log{n})^{1/\beta}}\]
	for $n\geq2$. 
\end{theorem}
\begin{proof}
	Since $\mathbb{P}$ is $\nu$-sparse for every $\nu>0$ and since $\vartheta>0$, we can use both parts of Theorem~\ref{thm:approxmain} to conclude that
	\[(a_n(\mathscr{C}_\varphi))^2 \asymp \|\mathscr{C}_\varphi e_n \|_{\mathscr{H}^2}^2 = n^{-2\vartheta} \prod_{j=1}^\infty \int_{-\pi}^\pi n^{-\frac{2}{j^\beta}(1-\cos{\theta_j})}\,\frac{d\theta_j}{2\pi}.\]
	We need to estimate the integrals
	\[I_{j,\beta}(n) = \int_{-\pi}^\pi n^{-\frac{2}{j^\beta}(1-\cos{\theta_j})}\,\frac{d\theta_j}{2\pi}\]
	for $n\geq2$. Let $J=\lfloor (\log{n})^{1/\beta} \rfloor$. When $j>J$ we estimate roughly to obtain that $n^{-4/j^\beta} \leq I_{j,\beta}(n) \leq 1$. Hence 
	\begin{equation}\label{eq:bigj} 
		\exp\left(-\frac{4}{\beta-1} (\log{n})^{1/\beta}\right) \leq \prod_{j>J} I_{j,\beta}(n) \leq 1. 
	\end{equation}
	Next we turn to $1 \leq j \leq J$, applying Lemma~\ref{lem:hankel} with $x=(\log{n})/j^\beta$ to see that 
	\begin{equation}\label{eq:smallj} 
		\left(\frac{1}{\pi \sqrt{2e}}\right)^J \prod_{j=1}^J \sqrt{\frac{j^\beta}{\log{n}}} \leq \prod_{j=1}^J I_{j,\beta}(n) \leq \left(\frac{\sqrt{\pi}}{4}\right)^J \prod_{j=1}^J \sqrt{\frac{j^\beta}{\log{n}}}. 
	\end{equation}
	From Stirling's formula we find that
	\[\prod_{j=1}^J \sqrt{\frac{j^\beta}{\log{n}}} \asymp \exp\left(\frac{\beta}{2}\left(\left(J+\frac{1}{2}\right)\log(J)-J\right)- \frac{J}{2}\log\log{n}\right).\]
	That is, since $J=\lfloor (\log{n})^{1/\beta} \rfloor$, 
	\begin{equation}\label{eq:prod} 
		\prod_{j=1}^J \sqrt{\frac{j^\beta}{\log{n}}} \asymp \exp\left(-\frac{\beta}{2}(\log{n})^{1/\beta} + \frac{\log\log{n}}{4} \right). 
	\end{equation}
	Combining \eqref{eq:bigj}, \eqref{eq:smallj}, and \eqref{eq:prod}, noting that $\frac{1}{\pi \sqrt{2e}} < \frac{\sqrt{\pi}}{4} < 1$, yields the desired estimates. 
\end{proof}

\section{Schatten classes} \label{sec:schatten} 
For $1\leq q<\infty$, a linear operator $T$ on a Hilbert space $H$ belongs to the Schatten class $S_q$ if $(a_n(T))_{n\geq1} \in \ell^q$, in which case its Schatten norm is given by 
\[\|T\|_{S_q}^q = \sum_{n=1}^\infty |a_n(T)|^q.\]
Let $(x_n)_{n\geq1}$ be an orthonormal basis of $H$. If $T \in S_q$, then 
\begin{equation}\label{eq:Spnec} 
	\sum_{n=1}^\infty |\langle T x_n, x_n \rangle|^q \leq \|T\|_{S_q}^q. 
\end{equation}
On the other hand, if $1 \leq q \leq 2$ and $\sum_{n\geq1} \|Tx_n\|^q < \infty$, then $T \in S_q$, and 
\begin{equation}\label{eq:Spsuf} 
	\|T\|_{S_q}^q \leq \sum_{n=1}^\infty \|Tx_n\|^q. 
\end{equation}
The necessary and sufficient conditions \eqref{eq:Spnec} and \eqref{eq:Spsuf} for Schatten membership can be found for example in \cite[pp.~94--95]{GK69}.

Let us now return to composition operators $\mathscr{C}_\varphi$ with symbols $\varphi \in \mathscr{G}_{\geq 1}$. Note from Theorem~\ref{thm:BQSest} that if $\vartheta>1/q$ for some $1 \leq q<\infty$, then $\mathscr{C}_\varphi \in S_q$. We can use \eqref{eq:Spnec} to obtain the following converse.
\begin{theorem}\label{thm:pline} 
	Let $\varphi \in \mathscr{G}_{\geq1}$. If $\mathscr{C}_\varphi \in S_q$ for some $1 \leq q < \infty$, then $\vartheta \geq 1/q$. 
\end{theorem}
\begin{proof}
	As in the proof of Theorem~\ref{thm:BQSest}, we exploit that $\varphi-\vartheta$ also belongs to $\mathscr{G}_{\geq1}$; by Lemma~\ref{lem:norm1} we have that $\|\mathscr{C}_{\varphi-\vartheta}^\ast\|=\|\mathscr{C}_{\varphi-\vartheta}\|=1$. Hence the ideal property \eqref{eq:ideal} implies that $\|\mathscr{C}_{\varphi-\vartheta}^\ast \mathscr{C}_\varphi\|_{S_q} \leq \|\mathscr{C}_\varphi\|_{S_q}$. We then apply \eqref{eq:Spnec} with $T=\mathscr{C}_{\varphi - \vartheta}^\ast \mathscr{C}_{\varphi}$ and $x_n = e_n$ as the standard basis of $\mathscr{H}^2$ to conclude that
	\[\sum_{n=1}^\infty \big|\langle \mathscr{C}_{\varphi - \vartheta}^\ast \mathscr{C}_{\varphi}e_n, e_n \rangle_{\mathscr{H}^2}\big|^q < \infty.\]
	Observing that
	\[\langle \mathscr{C}_{\varphi - \vartheta}^\ast \mathscr{C}_{\varphi}e_n, e_n \rangle_{\mathscr{H}^2} = \langle \mathscr{C}_{\varphi}e_n, \mathscr{C}_{\varphi - \vartheta}e_n \rangle_{\mathscr{H}^2} = \int_{\mathbb{T}^\infty} n^{\vartheta -2 \mre \varphi_0^\ast(\chi)} \,d\mu_\infty(\chi) ,\]
	a measure-theoretic argument then shows that $2 \mre \varphi_0^\ast(\chi) \geq \vartheta + 1/q$ for almost every $\chi \in \mathbb{T}^\infty$. Since $\essinf_{\chi\in\mathbb{T}^\infty} \mre \varphi_0^\ast(\chi) = \vartheta$ by \eqref{eq:varthetaast}, we conclude that $\vartheta \geq 1/q$. 
\end{proof}

Therefore only the case $\vartheta = 1/q$ is of further interest. In this setting, we have the following corollary of Theorem~\ref{thm:approxmain}. 
\begin{corollary}\label{cor:schatten} 
	Let $1 \leq q < \infty$. Suppose that $\varphi \in \mathscr{G}_{\geq1}$ with $\vartheta = 1/q$ and that $\varphi_0$ is supported on a sparse set of prime numbers $\mathbb{P}$. 
	\begin{enumerate}
		\item[(a)] If $1\leq q \leq 2$, then $\mathscr{C}_\varphi \in S_q$ if and only if $\big(\|\mathscr{C}_\varphi e_n\|_{\mathscr{H}^2}\big)_{n\geq1}\in\ell^q$.
		\item[(b)] If $2<q<\infty$ and $\mathbb{P}$ is $2/(2+q)$-sparse, then again we have that $\mathscr{C}_\varphi \in S_q$ if and only if $\big(\|\mathscr{C}_\varphi e_n\|_{\mathscr{H}^2}\big)_{n\geq1}\in\ell^q$.
	\end{enumerate}
\end{corollary}
\begin{proof}
	If $1 \leq q \leq 2$, then the statement is a consequence of Theorem~\ref{thm:approxmain} (a) and the general sufficient condition \eqref{eq:Spsuf}. If $2<q<\infty$, then the statement follows directly from Theorem~\ref{thm:approxmain}. 
\end{proof}

When $q=2$, Theorem~\ref{thm:pline} has previously been observed by Finet, Queff\'{e}lec and Volberg \cite[p.~279]{FQV04}. We finish this section by pointing out the following curiosity of the Hilbert--Schmidt case. If $\varphi \in \mathscr{G}_{\geq1}$, then
\[\|\mathscr{C}_\varphi\|_{S_2}^2 = \sum_{n=1}^\infty \|\mathscr{C}_\varphi e_n\|_{\mathscr{H}^2}^2 = \sum_{n=1}^\infty \|\mathscr{C}_{\varphi_0} e_n\|_{\mathscr{H}^2}^2 = \|\mathscr{C}_{\varphi_0}\|_{S_2}^2.\]
In other words, $\mathscr{C}_\varphi$ is Hilbert--Schmidt if and only if $\mathscr{C}_{\varphi_0}$ is Hilbert--Schmidt. Note from the examples of Section~\ref{sec:affine} and Section~\ref{sec:angle} that the approximation numbers of $\mathscr{C}_{\varphi_0}$ tend to have much more rapid decay than those of $\mathscr{C}_\varphi$, even when $\vartheta = 1/2$. This is compensated for by the fact that if $c_0\geq1$, then $a_1(\mathscr{C}_\varphi)=\|\mathscr{C}_\varphi\|=1$, while it always holds that $\|\mathscr{C}_\varphi\| > 1$ when $c_0=0$.

\section{Restricted counting functions} \label{sec:restricted} 
We shall now begin to work towards Theorem~\ref{thm:compact}. For notational simplicity we will assume that $p=2$ throughout. We therefore consider symbols
\[\varphi(s)=c_0s + \varphi_0(s)\]
where $c_0\geq1$ and $\varphi_0(s)=\Phi(2^{-s})$ for some analytic function $\Phi \colon \mathbb{D}\to\mathbb{C}_0$. Let $\mathbb{O} = \mathscr{M}(\{2\}^\perp)$ denote the set of odd numbers. As in Section~\ref{sec:orth}, $\mathscr{H}^2$ orthogonally decomposes into the subspaces $\mathscr{H}_j^2$, for $j \in \mathbb{O}$. Each subspace $\mathscr{H}_j^2$ is comprised of elements $f(s) = j^{-s} F(2^{-s})$, where $F$ is in the Hardy space $H^2(\mathbb{D})$ of the unit disc. Note that $\mathscr{H}^2_j$ is a space of absolutely convergent Dirichlet series in $\mathbb{C}_0$, while the elements of $\mathscr{H}^2$ are generally only convergent in the smaller half-plane $\mathbb{C}_{1/2}$. Moreover, the absolute values of functions in $\mathscr{H}^2_j$ are periodic: If $s \in \mathbb{C}_0$, then 
\begin{equation}\label{eq:periodic} 
	|f(s+2\pi i /\log{2})| = |f(s)|. 
\end{equation}
From \eqref{eq:periodic} it is easy to establish the Carlson--type formula 
\begin{equation}\label{eq:carlson} 
	\|f\|_{\mathscr{H}^2}^2 = \lim_{\sigma \to 0^+} \frac{\log{2}}{2\pi} \int_{-\frac{\pi}{\log{2}}}^{\frac{\pi}{\log{2}}} |f(\sigma+it)|^2 \,dt, 
\end{equation}
valid for all $f \in \mathscr{H}^2_j$ and $j \in \mathbb{O}$. 

From \eqref{eq:carlson} and a direct computation we obtain the following Littlewood--Paley formula. The proof is very similar to those of \cite[Prop.~2]{Bayart02} and \cite[Lem.~2.2]{BP21}.
\begin{lemma}\label{lem:LP2} 
	Suppose that $f \in \mathscr{H}^2_j$ for some $j \in \mathbb{O}$. Then 
	\begin{equation}\label{eq:LP2} 
		\|f\|_{\mathscr{H}^2}^2 = |f(+\infty)|^2 + \frac{2\log{2}}{\pi}\int_{-\frac{\pi}{\log{2}}}^{\frac{\pi}{\log{2}}} \int_0^\infty |f'(\sigma+it)|^2 \,\sigma\,d\sigma dt. 
	\end{equation}
\end{lemma}
\begin{proof}
	Swap the order of integration in \eqref{eq:LP2} and then apply \eqref{eq:carlson} for each $\sigma>0$. The proof is finished by using the formula 
	\begin{equation}\label{eq:coeffint} 
		4 \int_0^\infty n^{-2\sigma} \,\sigma d\sigma = \frac{1}{(\log{n})^2}. \qedhere 
	\end{equation}
\end{proof}

Define the restricted counting function $\mathscr{N}_\varphi$ by 
\begin{equation}\label{eq:counting2} 
	\mathscr{N}_\varphi(w) = \sum_{\substack{s \in \varphi^{-1}(\{w\}) \\ -\pi/\log{2} \leq \mim{s} < \pi/\log{2}}} \mre{s}. 
\end{equation}
For technical reasons, we have included $\mim{s}=-\pi/\log{2}$ in the definition of $\mathscr{N}_\varphi$ (see \eqref{eq:Ndecomp} below). This only affects the value of $\mathscr{N}_\varphi$ on a set of measure zero. The following version of the Stanton formula follows immediately from Lemma~\ref{lem:LP2} and a change of variables.
\begin{lemma}\label{lem:CoV} 
	Suppose that $\varphi \in \mathscr{G}_{\geq1}$ and that $\varphi_0$ is supported on $\mathbb{P}=\{2\}$. If $f \in \mathscr{H}^2_j$ for some $j\in\mathbb{O}$, then
	\[\|\mathscr{C}_\varphi f\|_{\mathscr{H}^2}^2 = |f(+\infty)|^2 + \frac{2\log{2}}{\pi} \int_{\mathbb{C}_0} |f'(w)|^2 \, \mathscr{N}_\varphi(w)\,dw.\]
\end{lemma}
\begin{proof}
	We can apply Lemma~\ref{lem:LP2}, since $\mathscr{C}_\varphi f \in \mathscr{H}^2_{j^{c_0}}$ by Lemma~\ref{lem:Csum}. One obtains the desired formula by making the non-injective change of variable $w = \varphi(s)$ in the usual manner (see e.g.~\cite[Sec.~10.3]{Shapiro93}). Note also that $\varphi(+\infty) = +\infty$. 
\end{proof}

Our next goal is to obtain a version of the Littlewood--type inequality \eqref{eq:NLest} for the restricted counting function. Our main tool will be the classical Littlewood inequality for the Hardy space of the unit disc (see e.g.~\cite[Sec.~10.4]{Shapiro93}). Recall that if $\psi$ is an analytic self-map of the unit disc $\mathbb{D}$, then the Nevanlinna counting function is defined as
\[N_\psi^{\mathbb{D}}(\eta) = \sum_{z \in \psi^{-1}(\{\eta\})} \log{\frac{1}{|z|}}\]
for $\eta \in \mathbb{D}\setminus\{\psi(0)\}$. The Littlewood--type inequality for $N_\psi^{\mathbb{D}}$ takes the form 
\begin{equation}\label{eq:DLest} 
	N_\psi^{\mathbb{D}}(\eta) \leq \log\left|\frac{1-\overline{\eta}\psi(0)}{\eta-\psi(0)}\right| 
\end{equation}
for $\eta \neq \varphi(0)$. The proof of the following result is inspired by \cite[Lem.~2.4]{BP21}.
\begin{lemma}\label{lem:IAmTheta} 
	Suppose that $\varphi \in \mathscr{G}_{\geq1}$ and that $\varphi_0$ is supported on $\mathbb{P}=\{2\}$. There is a constant $C=C(\varphi)$ such that
	\[\mathscr{N}_\varphi(w) \leq C \frac{\mre{w}}{1+(\mim{w})^2}\]
	for $0<\mre{w} \leq c_0\pi/\log{2}$. 
\end{lemma}
\begin{proof}
	Let $\Theta \colon \mathbb{D} \to S$ denote the Riemann map of $\mathbb{D}$ onto the half-strip
	\[S = \{s = \sigma+it \,:\, \sigma>0, -1<t<1\},\]
	normalized so that $\Theta(0) = 1$, $\Theta'(0) > 0$, and for $T > 0$, let $\Theta_T = T \Theta$. For any $w \in \mathbb{C}_0$ and any $T>0$, the function
	\[\psi(z) = \frac{\varphi(\Theta_T(z))-w}{\varphi(\Theta_T(z))+\overline{w}}\]
	is an analytic map from $\mathbb{D}$ to $\mathbb{D}$. Fix a number $T \geq 2\mre{w}/c_0$. Then
	\[\mre{\varphi(\Theta_T(0))}= \mre{\varphi(T)} \geq 2\mre{w},\]
	since $\mre{\varphi_0(s)}\geq0$ for every $s \in \mathbb{C}_0$. Hence it is evident that $\psi(0)\neq0$. Using \eqref{eq:DLest} with $\eta=0$ we find that
	\[\sum_{z \in \psi^{-1}(\{0\})} \log{\frac{1}{|z|}} \leq \log{\frac{1}{|\psi(0)|}}.\]
	Noting that $\psi(z) = 0$ if and only if $\varphi(\Theta_T(z)) = w$, and substituting $s = \Theta_T(z)$, we rewrite this estimate as 
	\begin{equation}\label{eq:RL1} 
		\sum_{s \in \varphi^{-1}(\{w\})} \log{\frac{1}{|\Theta_T^{-1}(s)|}} \leq \log\left|\frac{\varphi(T)+\overline{w}}{\varphi(T)-w}\right|. 
	\end{equation}
	By standard regularity results for conformal maps there is an absolute constant $C > 0$ such that
	\[\mre s \leq C T \log{\frac{1}{|\Theta^{-1}_T(s)|}}\]
	whenever $|\mim{s}| \leq T/2$ and $0<\mre{s} \leq T/2$. Since $\mre{\varphi_0(s)\geq0}$ for every $s\in\mathbb{C}_0$, we of course have that if $\varphi(s)=w$, then $\mre{s}\leq \mre{w}/c_0$. This implies that if $T \geq 2\max(\mre{w}/c_0,\pi/\log{2})$, then 
	\begin{equation}\label{eq:RL2} 
		\mathscr{N}_\varphi(w) \leq \sum_{\substack{s \in \varphi^{-1}(\{w\}) \\ |\mim{s}| \leq T/2 \\ 0<\mre{s} \leq T/2}} \mre{s} \leq C T \sum_{s \in \varphi^{-1}(\{w\})} \log{\frac{1}{|\Theta_T^{-1}(s)|}}. 
	\end{equation}
	Noting that both $\varphi(T)$ and $w$ are in the half-plane $\mathbb{C}_0$, a basic estimate for the pseudo-hyperbolic metric (see e.g.~\cite[Lem.~2.3]{BP21}) yields that 
	\begin{equation}\label{eq:RL3} 
		\log\left|\frac{\varphi(T)+\overline{w}}{\varphi(T)-w}\right| \leq \frac{2\mre{\varphi(T)}\mre{w}}{|\varphi(T)-w|^2}. 
	\end{equation}
	Combining \eqref{eq:RL1}, \eqref{eq:RL2}, and \eqref{eq:RL3}, and recalling that $\mre{\varphi(T)} \geq 2\mre{w}$, we conclude that
	\[\mathscr{N}_\varphi(w) \leq 2C T \frac{\mre{\varphi(T)}\mre{w}}{\frac{(\mre{\varphi(T)})^2}{4}+\left(\mim{\varphi(T)}-\mim{w}\right)^2,}\]
	as long as $T \geq 2\max(\mre{w}/c_0,\pi/\log{2})$. For the purposes of the present lemma, where we restrict our attention to $0<\mre{w} \leq c_0\pi/\log{2}$, it is sufficient to choose $T = 2\pi/\log{2}$. 
\end{proof}

While Theorem~\ref{thm:compact} is stated in terms of the Nevanlinna counting function $N_\varphi$, we shall prove it via the restricted counting function $\mathscr{N}_\varphi$, which is natural in view of Lemma~\ref{lem:CoV}. To bridge the gap, the remainder of this section is devoted to the following result.
\begin{theorem}\label{thm:Nevandecomp} 
	Suppose that $\varphi \in \mathscr{G}_{\geq1}$ and that $\varphi_0$ is supported on $\mathbb{P}=\{2\}$. Let $N_\varphi$ denote the Nevanlinna counting function \eqref{eq:nevanlinna} and $\mathscr{N}_\varphi$ the restricted counting function \eqref{eq:counting2}. Then
	\[\lim_{\mre{w}\to 0^+} \frac{N_\varphi(w)}{\mre{w}} = 0 \qquad \iff \qquad \lim_{\mre{w}\to 0^+} \frac{\mathscr{N}_\varphi(w)}{\mre{w}}=0.\]
\end{theorem}
Three preliminary results are required for the proof of Theorem~\ref{thm:Nevandecomp}. We first decompose the Nevanlinna counting function as 
\begin{equation}\label{eq:Ndecomp} 
	N_\varphi(w) = \sum_{k\in \mathbb{Z}} \mathscr{N}_{\varphi,k}(w), 
\end{equation}
where $\mathscr{N}_{\varphi,k}(w)$ denotes the restricted counting function
\[\mathscr{N}_{\varphi,k}(w) = \sum_{s \in \varphi^{-1}(\{w\}) \cap S_k} \mre{s}\]
corresponding to the the half-strip
\[S_k = \left\{s \in \mathbb{C}_0 \,:\, -\frac{\pi}{\log{2}} \leq \mim{s}- \frac{2\pi k}{\log{2}} < \frac{\pi}{\log{2}}\right\}.\]
Note in particular that $\mathscr{N}_{\varphi,0}=\mathscr{N}_\varphi$. We begin our study of $\mathscr{N}_{\varphi,k}$ with the following Littlewood--type estimate.
\begin{lemma}\label{lem:IAmTheta2} 
	Suppose that $\varphi \in \mathscr{G}_{\geq1}$ and that $\varphi_0$ is supported on $\mathbb{P}=\{2\}$. There is a constant $C=C(\varphi)> 0$ such that the estimate
	\[\mathscr{N}_{\varphi,k}(w) \leq C \frac{\mre w}{1 + \big|\mim{w}-c_0 \frac{2\pi k}{\log{2}}\big|^2}\]
	holds uniformly for for $0<\mre{w} \leq c_0\pi/\log{2}$ and $k\in\mathbb{Z}$. 
\end{lemma}
\begin{proof}
	If $\varphi(w)=s$ for some $s \in S_k$, then the periodicity of $\varphi_0$ implies that
	\[\varphi\left(s-\frac{2\pi i}{\log{2}} k \right) = w - c_0 \frac{2\pi i}{\log{2}} k = \widetilde{w}.\]
	Hence we get from Lemma~\ref{lem:IAmTheta} that
	\[\mathscr{N}_{\varphi,k}(w) = \mathscr{N}_{\varphi,0}\left(\widetilde{w}\right) \leq C \frac{\mre{\widetilde{w}}}{1+|\mim{\widetilde{w}}|^2}. \qedhere\]
\end{proof}

For $0<\mre{w}\leq c_0 \pi/\log{2}$, we can combine Lemma~\ref{lem:IAmTheta2} and \eqref{eq:Ndecomp} to obtain a less precise version of Bayart's estimate \eqref{eq:NLest}. Specifically,
\[N_\varphi(w) = \sum_{k\in \mathbb{Z}} \mathscr{N}_{\varphi,k}(w) \leq C\sum_{k\in \mathbb{Z}} \frac{\mre w}{1 + \big|\mim{w}-c_0 \frac{2\pi k}{\log{2}}\big|^2} \leq \widetilde{C} \mre{w}.\]
It is reasonable to expect that the main contribution to $N_\varphi(w)$ in the decomposition \eqref{eq:Ndecomp} arises from the $k$ such that $w/c_0 \in S_k$. This is the main idea in the proof of Theorem~\ref{thm:Nevandecomp}. Before proceeding, we record the following simple fact.
\begin{lemma}\label{lem:shiftit} 
	Suppose that $\varphi \in \mathscr{G}_{\geq1}$ and that $\varphi_0$ is supported on $\mathbb{P}=\{2\}$. Let $j_1,j_2,k_1,k_2$ be integers satisfying the equation $k_1-j_1 = k_2 -j_2$. For every $w_1$ such that $w_1/c_0 \in S_{j_1}$, there is a $w_2$ such that $w_2/c_0 \in S_{j_2}$, $\mre w_2 = \mre w_1$, and
	\[\mathscr{N}_{\varphi,k_1}(w_1)= \mathscr{N}_{\varphi,k_2}(w_2).\]
\end{lemma}
\begin{proof}
	Given $w_1$, define
	\[w_2 = w_1 + c_0 (j_2-j_1) \frac{2\pi i }{\log{2}}.\]
	Clearly $w_2/c_0 \in S_{j_2}$ and $\mre{w_2}=\mre{w_1}$. Consider $s_1 \in S_{k_1}$ such that $\varphi(s_1)=w_1$. Define
	\[s_2 = s_1+(k_2-k_1)\frac{2\pi i}{\log{2}}.\]
	If $k_1-j_1= k_2 - j_2$, then $k_2-k_1 = j_2-j_1$, and thus $\varphi(s_2) = w_2$ (with the same multiplicity as $\varphi(s_1)=w_1$), by the periodicity of $\varphi_0$. This demonstrates that
	\[\mathscr{N}_{\varphi,k_1}(w_1) \leq \mathscr{N}_{\varphi,k_2}(w_2),\]
	but by symmetry we also have the reverse estimate. 
\end{proof}

A direct consequence of Lemma~\ref{lem:shiftit} is that the property $\mathscr{N}_{\varphi,k}(w) = o(\mre w)$ does not depend on $k$.
\begin{lemma}\label{lem:Nevank} 
	Suppose that $\varphi \in \mathscr{G}$ with $c_0 \geq 1$ and that $\varphi_0$ is supported on $\mathbb{P}=\{2\}$. Fix $k \in \mathbb{Z}$. Then
	\[\lim_{\mre{w}\to 0^+} \frac{\mathscr{N}_{\varphi,k}(w)}{\mre{w}} = 0 \qquad \iff \qquad \lim_{\mre{w}\to 0^+} \frac{\mathscr{N}_\varphi(w)}{\mre{w}}=0.\]
\end{lemma}

We now prove the main result of this section.
\begin{proof}
	[Prood of Theorem~\ref{thm:Nevandecomp}] The implication $\implies$ is trivial since $\mathscr{N}_\varphi(w) \leq N_\varphi(w)$. For the converse implication $\impliedby$, we assume that
	\[\lim_{j \to \infty} \frac{\mathscr{N}_\varphi(w_j)}{\mre{w_j}}=0\]
	for every sequence $(w_j)_{j\geq1}$ in $\mathbb{C}_0$ such that $\mre{w_j}\to 0$. We may without loss of generality assume that $0<\mre{w_j} \leq c_0\pi/\log{2}$, so that Lemma~\ref{lem:IAmTheta2} applies.
	
	Fix $\varepsilon>0$. We need to prove that there is some $J >0$ such that 
	\begin{equation}\label{eq:epsest} 
		\frac{N_\varphi(w_j)}{\mre{w_j}} \leq \varepsilon 
	\end{equation}
	for every $j\geq J$. For each $j\geq1$, define $k_j$ by the requirement that $w_j/c_0 \in S_{k_j}$. By Lemma~\ref{lem:IAmTheta2} and the decomposition \eqref{eq:Ndecomp} there is some non-negative integer $K$ (which does not depend on $j$) such that
	\[\frac{N_\varphi(w_j)}{\mre{w_j}} \leq \sum_{|k-k_j| \leq K} \frac{\mathscr{N}_{\varphi,k}(w_j)}{\mre{w_j}} + \frac{\varepsilon}{2} = \sum_{|k| \leq K} \frac{\mathscr{N}_{\varphi,k}(\widetilde{w}_j)}{\mre{\widetilde{w}_j}} + \frac{\varepsilon}{2},\]
	where the points $\widetilde{w}_j$ arise from Lemma~\ref{lem:shiftit}, whence $\mre{\widetilde{w}_j} = \mre{w_j} \to 0$ as $j\to \infty$. We can now appeal to Lemma~\ref{lem:Nevank} to choose $J$ so large that
	\[\frac{\mathscr{N}_{\varphi,k}(\widetilde{w}_j)}{\mre{\widetilde{w}_j}} \leq \frac{\varepsilon}{4K+2}\]
	whenever $|k|\leq K$ and $j\geq J$. Hence \eqref{eq:epsest} holds for $j\geq J$. 
\end{proof}

\section{Proof of Theorem~\ref{thm:compact}} \label{sec:cpctproof} 
By Theorem~\ref{thm:Nevandecomp}, it is sufficient to prove that $\mathscr{C}_\varphi \colon \mathscr{H}^2 \to \mathscr{H}^2$ is compact if and only if $\mathscr{N}_\varphi(w) = o(\mre w)$ as $\mre{w}\to0^+$. To do so, we will adapt the proof of \cite[Thm.~1.4]{BP21} to the present case, in an argument that relies on Lemma~\ref{lem:CoV} and Lemma~\ref{lem:IAmTheta}. Several preliminary results are required; the first two estimates are similar to those of \cite[Lem.~7.1]{BP21} and \cite[Lem.~7.3]{BP21}, respectively. 
\begin{lemma}\label{lem:weakly} 
	Fix $j\in \mathbb{O}$. Suppose that $(f_k)_{k\geq1}$ is a sequence in $\mathscr{H}^2_j$ such that $\|f_k\|_{\mathscr{H}^2} \leq 1$ for every $k\geq1$ and which converges weakly to $0$. For every $\varepsilon>0$ and $\theta>0$ there is some constant $K=K(\varepsilon,\theta)$ such that 
	\begin{equation}\label{eq:weakly} 
		|f_k(+\infty)|^2 + \frac{2\log{2}}{\pi} \int_{\mre{w}\geq \theta} |f_k'(w)|^2 \, \mathscr{N}_\varphi(w)\,dw \leq \varepsilon^2 
	\end{equation}
	whenever $k\geq K$. 
\end{lemma}
\begin{proof}
	We consider first the case $j=1$ and write $f_k(s) = \sum_{l\geq0} b_l(k) 2^{-ls}$. The assumption that $(f_k)_{k\geq1}$ converges weakly to $0$ implies that $|b_l(k)| \to 0$ as $k\to\infty$, for every fixed $l$. Since $f_k(+\infty)=b_0(k)$, there is some $K_1$ such that if $k\geq K_1$ then $|f_k(+\infty)| \leq \varepsilon/\sqrt{2}$. Next we want to demonstrate that there is a constant $K_2=K_2(\varepsilon,\theta,j)$ such that 
	\begin{equation}\label{eq:theguy} 
		|f_k'(w)| \leq \frac{\varepsilon}{\sqrt{2}} \, |e_2'(w)| 
	\end{equation}
	for $\mre{w}\geq \theta>$0 whenever $k\geq K_2$. Here we recall that $e_2(s)=2^{-s}$. This would give the stated claim since then
	\[\frac{2\log{2}}{\pi} \int_{\mre{w}\geq \theta} |f_k'(w)|^2 \, \mathscr{N}_\varphi(w)\,dw \leq \frac{\varepsilon^2}{2} \frac{2\log{2}}{\pi} \int_{\mre{w}\geq 0} |e_2'(w)|^2 \, \mathscr{N}_\varphi(w)\,dw \leq \frac{\varepsilon^2}{2}\]
	whenever $k\geq K_2$, where we used Lemma~\ref{lem:CoV} in the final estimate. To establish the estimate \eqref{eq:theguy}, we first choose $L=L(\theta,\varepsilon)\geq2$ so large that
	\[\sum_{l=L}^\infty l 2^{-l\mre{w}} \leq \frac{\varepsilon}{2\sqrt{2}} 2^{-\mre{w}}\]
	whenever $\mre{w}\geq \theta>0$. We then choose $K_2=K_2(L,\varepsilon)$ so large that
	\[|b_l(k)| \leq \frac{\varepsilon}{2\sqrt{2}} \frac{2}{L(L-1)}\]
	for $l=1,\ldots, L-1$ and $k\geq K_2$. Since $\|f_k\|_{\mathscr{H}^2}\leq1$, we certainly have $|b_l(k)|\leq1$ for $l\geq L$. By the triangle inequality we obtain that
	\[|f_k'(w)| \leq \sum_{l=1}^{L-1} \frac{\varepsilon}{2\sqrt{2}} \frac{2}{L(L-1)} (\log{2^l}) 2^{-l\mre{w}} + \sum_{l=L}^\infty \log(2^l) 2^{-l\mre{w}} \leq \frac{\varepsilon}{\sqrt{2}} |e_2'(w)|\]
	whenever $k\geq K_2$. 
	
	If $j>1$, the same proof works with the following modifications. In this case $f_k(+\infty)=0$. The sum for $f_k'$ now starts at $l=0$, but taking this into account the same approach shows that
	\[|f_k'(w)| \leq \varepsilon |e_j'(w)|\]
	for $k\geq K(\varepsilon,\theta)$, yielding \eqref{eq:weakly} by Lemma~\ref{lem:CoV}. 
\end{proof}
\begin{lemma}\label{lem:embedding} 
	Fix $\delta>0$ and suppose that $f(s) = \sum_{n\geq1} a_n n^{-s}$ is in $\mathscr{H}^2_j$ for some $j \in \mathbb{O}$. Then there is a constant $C_\delta>0$ such that
	\[\int_{-\infty}^\infty |f'(\sigma+it)|^2 \,\frac{dt}{(1+t^2)^{(1+\delta)/2}} \leq C_\delta \sum_{n=1}^\infty |a_n|^2 (\log{n})^2 n^{-2\sigma}\]
	for every $\sigma>0$. 
\end{lemma}
\begin{proof}
	Since $f$ is in $\mathscr{H}^2_j$, we can use the periodicity condition \eqref{eq:periodic} and \eqref{eq:carlson} to see that
	\[\int_{\frac{2\pi i k}{\log{2}}}^{\frac{2\pi i (k+1)}{\log{2}}} |f'(\sigma+it)|^2 \,dt = \frac{2\pi}{\log{2}}\sum_{n=1}^\infty |a_n|^2 (\log{n})^2 n^{-2\sigma}\]
	for every $k \in \mathbb{Z}$ and every $\sigma>0$. The estimate easily follows.
\end{proof}

We shall also require the following pointwise estimate for the derivative of a function in $\mathscr{H}^2_j$.
\begin{lemma}\label{lem:pest} 
	Suppose that $j\in \mathbb{O}\setminus\{1\}$. If $f(s) = \sum_{k\geq0} b_k j^{-s} 2^{-ks}$ is in $\mathscr{H}_j^2$, then
	\[|f'(w)|^2 \leq C_\theta \|f\|_{\mathscr{H}^2}^2 (\log{j})^2 j^{-2\mre{w}}, \qquad C_\theta = \frac{1+4^{-\theta}}{(1-4^{-\theta})^3},\]
	for $\mre{w}\geq \theta>0$. 
\end{lemma}
\begin{proof}
	Applying the Cauchy--Schwarz inequality we find that 
	\begin{align*}
		|f'(w)|^2 &\leq \|f\|_{\mathscr{H}^2}^2 \sum_{k=0}^\infty \big(\log(j 2^k)\big)^2 j^{-2\mre{w}} 4^{-k\mre{w}} \\
		&= \|f\|_{\mathscr{H}^2}^2 (\log{j})^2 j^{-2\mre{w}} \sum_{k=0}^\infty \left(1+\frac{k\log{2}}{\log{j}}\right)^2 4^{-k\mre{w}} \\
		& \leq \|f\|_{\mathscr{H}^2}^2 (\log{j})^2 j^{-2\mre{w}} \sum_{k=0}^\infty (1+k)^2 4^{-k\theta} 
	\end{align*}
	where we in the final estimate used that $j\geq3$ and $\mre{w}\geq \theta$. 
\end{proof}
\begin{proof}
	[Proof of Theorem~\ref{thm:compact}: Sufficiency] We assume now that
	\[\lim_{\mre{w}\to0^+} \frac{N_\varphi(w)}{\mre{w}} =0\]
	where $N_\varphi$ is the Nevanlinna counting function \eqref{eq:nevanlinna}. By Theorem~\ref{thm:Nevandecomp} this is actually equivalent to the weaker assumption that 
	\begin{equation}\label{eq:restricted0} 
		\lim_{\mre{w}\to0^+} \frac{\mathscr{N}_\varphi(w)}{\mre{w}} =0, 
	\end{equation}
	where $\mathscr{N}_\varphi$ is the restricted counting function \eqref{eq:counting2}. Our goal is to prove that $\mathscr{C}_\varphi$ is compact on $\mathscr{H}^2$, which by Lemma~\ref{lem:Csum} is equivalent to proving that 
	\begin{enumerate}
		\item[(i)] $\mathscr{C}_{\varphi,j}$ is compact for every $j \in \mathbb{O}$, 
		\item[(ii)] $\|\mathscr{C}_{\varphi,j}\|\to 0$ as $j \to \infty$. 
	\end{enumerate}
	
	We begin by establishing an estimate that is of relevance to the proofs of both claims. Fix $0<\delta<1$. For every $\varepsilon>0$, there is some $0<\theta=\theta(\varepsilon) \leq c_0 \pi/\log{2}$ such that if $0<\mre{w}\leq \theta$, then 
	\begin{equation}\label{eq:epsiest} 
		\mathscr{N}_\varphi(w) \leq \varepsilon^2 \frac{\mre{w}}{(1+(\mim{w})^2)^{(1+\delta)/2}}. 
	\end{equation}
	Suppose for the sake of contradiction that \eqref{eq:epsiest} does not hold along some sequence $(w_k)_{k\geq1}$ in $\mathbb{C}_0$ with $c_0 \pi/\log{2} \geq \mre{w_k} \to 0$ as $k \to \infty$. If $|\mim{w_k}|$ is unbounded, we obtain a contradiction to Lemma~\ref{lem:IAmTheta}. However, if $|\mim{w_k}|$ is bounded we have a contradiction to \eqref{eq:restricted0}. 
	
	Combining \eqref{eq:epsiest} with Lemma~\ref{lem:embedding} we find that if $f(s) =\sum_{n\geq1} a_n n^{-s}$ is any function in $\mathscr{H}^2_j$ with $\|f\|_{\mathscr{H}^2}\leq1$, then 
	\begin{equation}\label{eq:epsiest2} 
		\int_{\mre{w} \leq \theta} |f'(w)|^2 \,\mathscr{N}_\varphi(w)\,dw \leq C_\delta \varepsilon^2 \sum_{n=1}^\infty |a_n|^2 (\log{n})^2 \int_0^\theta n^{-2\sigma} \sigma \,d\sigma \leq C_\delta \varepsilon^2, 
	\end{equation}
	where we made use of the identity \eqref{eq:coeffint} to assert the final inequality.
	
	Let us prove the validity of (i). Fix $j \in \mathbb{O}$ and suppose that $(f_k)_{k\geq1}$ is a sequence in $\mathscr{H}^2_j$ which converges weakly to $0$ and satisfies $\|f_k\|_{\mathscr{H}^2} \leq 1$ for every $k$. We then choose $\theta>0$ such that \eqref{eq:epsiest2} holds for each $f_k$. Next, appealing to Lemma~\ref{lem:weakly}, we choose $K$ such that 
	\begin{equation}\label{eq:est3} 
		|f_k(+\infty)|^2 + \frac{2\log{2}}{\pi} \int_{\mre{w}\geq \theta} |f_k'(w)|^2 \, \mathscr{N}_\varphi(w)\,dw \leq \varepsilon^2 
	\end{equation}
	for $k \geq K$. By Lemma~\ref{lem:CoV}, combining \eqref{eq:epsiest2} and \eqref{eq:est3} yields that
	\[\|\mathscr{C}_\varphi f_k\|_{\mathscr{H}^2}^2 \leq (1 + C_\delta)\varepsilon^2\]
	when $k\geq K$. Since $\varepsilon$ was arbitrary, we conclude that $\|\mathscr{C}_\varphi f_k\|_{\mathscr{H}^2} \to 0$ as $k\to\infty$, and thus that $\mathscr{C}_{\varphi,j}$ is compact.
	
	The proof of (ii) is similar. Fix $\varepsilon>0$. Choosing $\theta > 0$ so that \eqref{eq:epsiest2} holds, it is by Lemma~\ref{lem:CoV} sufficient to find $J \geq 3$ such that 
	\begin{equation}\label{eq:getthis} 
		\frac{2\log{2}}{\pi} \int_{\mre{w}\geq\theta} |f_j'(w)|^2\,\mathscr{N}_\varphi(w)\,dw \leq \varepsilon 
	\end{equation}
	whenever $f_j \in \mathscr{H}^2_j$, $\|f_j\|_{\mathscr{H}^2} = 1$, and $j\geq J$. Using Lemma~\ref{lem:pest} and that $j\geq3$, we have that 
	\begin{align*}
		\int_{\mre{w}\geq\theta} |f_j'(w)|^2\,\mathscr{N}_\varphi(w)\,dw &\leq C_\theta \int_{\mre{w} \geq \theta} |e_j'(w)|^2\,\mathscr{N}_\varphi(w)\,dw \\
		&\leq C_\theta \frac{(\log{j})^2}{(\log{2})^2} \left(\frac{j}{2}\right)^{-2\theta} \int_{\mre{w}\geq\theta} |e_2'(w)|^2\,\mathscr{N}_\varphi(w)\,dw. 
	\end{align*}
	The final integral is less than $\pi/(2\log{2}) \|\mathscr{C}_\varphi e_2\|_{\mathscr{H}^2}^2 \leq \pi/(2\log{2})$ by Lemma~\ref{lem:CoV} and Lemma~\ref{lem:norm1}. We thus obtain \eqref{eq:getthis} for sufficiently large $j$. 
\end{proof}

For the proof of necessity, we require the following sub-mean value property of the Nevanlinna counting function for the unit disc. It is convenient to introduce the notation $\mathbb{D}(w,r)=\{\xi \in \mathbb{C}\,:\, |\xi-w|<r\}$.
\begin{lemma}\label{lem:submean} 
	Suppose that $\psi$ is an analytic self-map of the unit disc $\mathbb{D}$ and let $N_\psi^{\mathbb{D}}$ denote the Nevanlinna counting function associated with $\psi$. If $g$ is an analytic map from a domain $\Omega$ to $\mathbb{D}$, $\mathbb{D}(w,r) \subseteq\Omega$, and $\psi(0) \not \in g(\mathbb{D}(w,r))$, then
	\[N_\psi^{\mathbb{D}}(g(w)) \leq \frac{1}{\pi r^2} \int_{\mathbb{D}(w,r)} N_\psi^{\mathbb{D}}(g(\xi))\,d\xi.\]
\end{lemma}
\begin{proof}
	For a short proof we refer to \cite[Sec.~4.6]{Shapiro87}. 
\end{proof}
\begin{proof}
	[Proof of Theorem~\ref{thm:compact}: Necessity] We assume now that $\mathscr{C}_\varphi$ is compact on $\mathscr{H}^2$ and seek to establish that
	\[\lim_{\mre{w}\to 0^+} \frac{N_\varphi(w)}{\mre{w}}=0\]
	where $N_\varphi$ is the Nevanlinna counting function \eqref{eq:nevanlinna}. In view of Theorem~\ref{thm:Nevandecomp}, we may equivalently establish that
	\[\lim_{\mre{w}\to 0^+} \frac{\mathscr{N}_\varphi(w)}{\mre{w}}=0,\]
	where $\mathscr{N}_\varphi$ is the restricted counting function \eqref{eq:counting2}. 
	
	Since $\mathscr{C}_\varphi$ is compact on $\mathscr{H}^2$, it is certainly compact on the subspace $\mathscr{H}^2_1$. We shall make use of a version of Lemma~\ref{lem:CoV} adapted to a larger half-strip. If we first write down the Littlewood--Paley formula \eqref{eq:LP2} with respect to the half-strip $|\mim s| < 2\pi/\log 2$, and then make a change of variables, we obtain the formula 
	\begin{equation}\label{eq:CoValt} 
		\|\mathscr{C}_\varphi f\|_{\mathscr{H}^2}^2 = |f(+\infty)|^2 + \frac{\log{2}}{\pi}\int_{\mathbb{C}_0} |f'(w)|^2\, \widetilde{\mathscr{N}}_\varphi(w)\,dw, 
	\end{equation}
	for every $f \in \mathscr{H}^2_1$. Here the counting function has been restricted to the larger strip, so that
	\[\widetilde{\mathscr{N}}_\varphi(w) = \sum_{\substack{s \in \varphi^{-1}(\{w\}) \\ |\mim{s}| < 2\pi/\log{2}}} \mre{s}.\]
	At every point in $w\in\mathbb{C}_0$, the subspace $\mathscr{H}^2_1$ has a reproducing kernel which is bounded in $\mathbb{C}_0$. A direct computation shows that the normalized reproducing kernel at $w \in \mathbb{C}_0$ is given by 
	\begin{equation}\label{eq:RPK1} 
		K_w(s) = \frac{\sqrt{1-2^{-2\mre{w}}}}{1-2^{-\overline{w}-s}}. 
	\end{equation}
	If $(w_k)_{k\geq1}$ is any sequence in $\mathbb{C}_0$ with $\mre{w_k}\to 0^+$, then $(K_{w_k})_{k\geq1}$ converges weakly to $0$ in $\mathscr{H}^2_1$, and thus $\|\mathscr{C}_\varphi K_{w_k}\| \to 0$ as $k \to \infty$, since $\mathscr{C}_\varphi$ is compact. From \eqref{eq:CoValt} and \eqref{eq:RPK1} we therefore conclude that 
	\begin{equation}\label{eq:altNtest} 
		\lim_{k\to\infty} \frac{1}{(\mre{w_k})^3} \int_{\mathbb{D}(w_k,\mre{w_k}/2)} \widetilde{\mathscr{N}}_\varphi(\xi)\,d\xi =0. 
	\end{equation}
	Let us for the moment restrict our attention to a single $w = w_k$, assuming without loss of generality that $0<\mre{w} \leq c_0 \pi/\log{2}$. As in the proof of Lemma~\ref{lem:IAmTheta}, we define
	\[\psi(z) = \frac{\varphi(\Theta_T(z))-w}{\varphi(\Theta_T(z))+\overline{w}},\]
	where we now fix $T=2\pi/\log{2}$. We also let $g(\xi) = (\xi-w)/(\xi+\overline{w})$, so that $g$ is a conformal map from $\mathbb{C}_0$ to $\mathbb{D}$. Clearly, $\psi(z)=g(\xi)$ if and only if $\varphi(\Theta_T(z))=\xi$. Hence 
	\begin{equation}\label{eq:NpsiD} 
		N_\psi^{\mathbb{D}}(g(\xi)) = \sum_{z\in \psi^{-1}(\{g(\xi)\})} \log{\frac{1}{|z|}} = \sum_{\substack{s \in \varphi^{-1}(\{\xi\}) \\ |\mim{s}|<T}} \log{\frac{1}{|\Theta_T^{-1}(s)|}}. 
	\end{equation}
	If $\xi \in \mathbb{D}(w,\mre{w}/2)$, then
	\[\mre{\xi}\leq (3/2)\mre{w} \leq c_0 (3/4) T.\]
	Since $\mre{\varphi_0(s)}\geq0$, we see that if $\varphi(s)=\xi$, then certainly $\mre{s} \leq (3/4) T$. Set $S_T = \{s \in \mathbb{C}_0 \,:\, |\mim{s}|<T\}$. If $s \in S_T$, then it follows from a Kellogg--Warschawski type theorem (see e.g.~\cite[Thm.~3.9]{Pommerenke92}) that
	\[\log{\frac{1}{|\Theta_T^{-1}(s)|}} \asymp 1 - |\Theta_T^{-1}(s)| \asymp |(\Theta_T^{-1})'(s)| \dist(s, \partial S_{T}) \ll |s^2 + T^2| \mre s \ll \mre s.\]
	It is of course also possible to establish this estimate by direct computation. Since $\mre{\varphi(T)}\geq c_0T > \mre{\xi}$ for every $\xi \in \mathbb{D}(w,\mre{w}/2)$, we can use the estimate together with Lemma~\ref{lem:submean} for $\Omega=\mathbb{C}_0$ to conclude that
	\[N_\psi^{\mathbb{D}}(g(w)) \leq \frac{4}{\pi (\mre{w})^2} \int_{\mathbb{D}(w,\mre{w}/2)} N_\psi^{\mathbb{D}}(g(\xi))\,d\xi \ll \frac{1}{(\mre{w})^2}\int_{\mathbb{D}(w,\mre{w}/2)} \widetilde{\mathscr{N}}_\varphi(\xi)\,d\xi.\]
	Next we recall that $g(w)=0$ and return to formula \eqref{eq:NpsiD}. If we restrict ourselves to solutions of $\varphi(s)=w$ satisfying $|\mim{s}|\leq \pi/\log{2}=T/2$, then \eqref{eq:RL2} says that
	\[\mathscr{N}_\varphi(w) \ll N_\psi^{\mathbb{D}}(0) = N_\psi^{\mathbb{D}}(g(w)).\]
	Applying these two estimates for every $w = w_k$, we have finally proven that \eqref{eq:altNtest} implies the desired conclusion,
	\[\lim_{k\to\infty} \frac{\mathscr{N}_\varphi(w_k)}{\mre{w_k}}=0. \qedhere\]
\end{proof}

\section{Approximation numbers for angle maps} \label{sec:angle}
As in Section~\ref{sec:restricted} and Section~\ref{sec:cpctproof}, we will assume without loss of generality that we are working with the prime number $p=2$. Fix a real number $0 < \alpha < 1$ and let
\[\Phi_\alpha(z) = \left(\frac{1-z}{1+z}\right)^\alpha.\]
The function $\Phi_\alpha$ is a univalent map from the unit disc $\mathbb{D}$ onto the angle 
\[A_\alpha = \left\{s \in \mathbb{C} \,:\, |\mim{s}| < \tan\left(\frac{\alpha \pi}{2}\right) \mre{s}\right\}.\]
The following result is an immediate consequence of Theorem~\ref{thm:approxmain}. 
\begin{corollary}\label{cor:angle1} 
	Fix a positive integer $c_0$ and real numbers $\vartheta>0$ and $0<\alpha<1$. Let $\varphi_{\alpha, \vartheta}(s) = c_0 s + \vartheta + \Phi_\alpha(2^{-s})$. Then $\varphi_{\alpha,\vartheta}$ is in $\mathscr{G}_{\geq 1}$ and
	\[a_n(\mathscr{C}_{\varphi_{\alpha, \vartheta}}) \asymp n^{-\vartheta} (\log{n})^{-\frac{1}{2\alpha}}\]
	for $n\geq2$. The implied constants depend on $\alpha$ and $\vartheta$. 
\end{corollary}
\begin{proof}
	It is evident that $\varphi_{\alpha,\vartheta}$ is in $\mathscr{G}_{\geq1}$. Appealing to both parts of Theorem~\ref{thm:approxmain} and computing the $\mathscr{H}^2$-norm with \eqref{eq:H2astnorm}, we get for $n \geq 2$ that
	\[a_n(\mathscr{C}_{\varphi_{\alpha, \vartheta}}) \asymp \|\mathscr{C}_{\varphi_{\alpha,\vartheta}} e_n \|_{\mathscr{H}^2} = n^{-\vartheta} \left(\int_{\mathbb{T}} n^{-2\mre{\Phi_\alpha(e^{i\theta})}}\,\frac{d\theta}{2\pi}\right)^{\frac{1}{2}} \asymp n^{-\vartheta} (\log{n})^{-\frac{1}{2\alpha}},\]
	where a straightforward estimate has been carried out in the final step. 
\end{proof}

Theorem~\ref{thm:approxmain} (a) no longer yields the correct behavior of the approximation numbers in the more intricate case when $\vartheta = 0$. To prove Theorem~\ref{thm:angle2}, we will instead base our analysis on the change of variables formula from Lemma~\ref{lem:CoV} and on estimates of the restricted counting function. 

We begin our study with a geometric analysis of $\varphi_\alpha$. In this section, we mildly modify our notation by letting $S=(0,\infty)\times(-\pi/\log{2},\pi/\log{2})$.
\begin{lemma}\label{lem:anglebeta} 
	There is $\beta=\beta(c_0,\alpha)$, $\alpha<\beta<1$, such that
	\[\varphi_\alpha(S) \subseteq A_\beta.\]
\end{lemma}
\begin{proof}
	In view of the maximum principle, it is sufficient to show that there is $\beta$ such that 
	\begin{equation}\label{eq:betachoice} 
		\varphi_\alpha(\partial S \setminus \{\pm i \pi/\log{2}\}) \subseteq A_\beta. 
	\end{equation}
	By explicit computation we have that 
	\begin{equation}\label{eq:phiacalc} 
		\begin{aligned}
			\mre{\varphi_\alpha(it)} &= \cos\left(\frac{\alpha \pi}{2}\right)\left|\tan\left(\frac{\log{2}}{2}t\right)\right|^\alpha, \\
			\mim{\varphi_\alpha(it)} &= c_0t+\operatorname{sign}(t)\sin\left(\frac{\alpha \pi}{2}\right) \left|\tan\left(\frac{\log{2}}{2}t\right)\right|^\alpha, 
		\end{aligned}
	\end{equation}
	for $|t| < \pi/\log{2}$. Hence there is $B < \infty$ such that $|\mim{\varphi_\alpha(t)}|/\mre{\varphi_\alpha(t)} \leq B$ for all such $t$, so that $\varphi_\alpha(i(-\pi/\log{2}, \pi/\log{2})) \subseteq A_{\beta_0}$ for $\beta_0=\frac{2}{\pi}\arctan(B)<1$. If $s = \sigma \pm i \pi/\log{2}$ for $\sigma > 0$, then
	\[\varphi_\alpha(s) = \pm c_0 i \pi/\log{2} + c_0 \sigma + \Phi_\alpha(-2^{-\sigma}),\]
	from which it is clear that $\varphi_{\alpha}(s) \mp c_0 i \pi/\log{2} \in [K, \infty)$, where
	\[K = \inf_{\sigma > 0} c_0 \sigma + \Phi_\alpha(-2^{-\sigma}) > 0.\]
	It follows that \eqref{eq:betachoice} holds if we choose $\beta \geq \beta_0$ sufficiently latge. 
\end{proof}

The next lemma is essentially a consequence of the fact that $\varphi_\alpha(S)$ looks very similar to $A_\alpha$ locally around $w = 0$. We refer to Figure~\ref{fig:alphabeta} for an illustration.

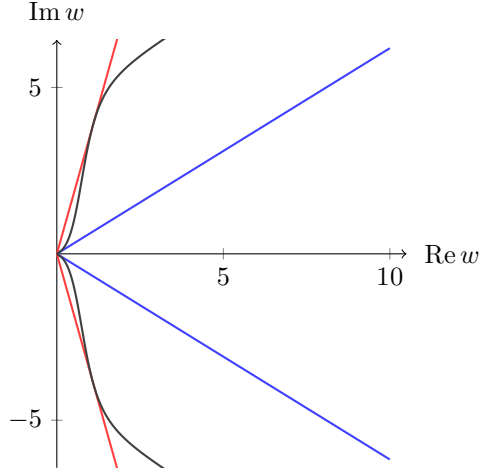
\begin{figure}
	\centering
	\begin{tikzpicture}
		\begin{axis}[
			axis equal image,
			xlabel= $\mre{w}$, 
			ylabel = $\mim{w}$,
			axis lines = center,
			xmin = 0,
			xmax = 10.5,
			ymin = -10*0.618-0.25,
			ymax = 10*0.618+0.25,
			xtick={5,10},
			ytick={-5,5},
			every axis x label/.style={
				at={(ticklabel* cs:1.025)},
				anchor=west,
			},
			every axis y label/.style={
				at={(ticklabel* cs:1.025)},
				anchor=south,
			},
			axis line style={->}]
			
			\pgfmathsetmacro{\alp}{0.35241638234}
			\pgfmathsetmacro{\bet}{0.8256615}
			\pgfmathsetmacro{\lims}{4.5323601418}
			\pgfmathsetmacro{\nlim}{10/cos(deg(\alp*pi/2))}
			
			\addplot [domain=0:\nlim, samples=10, blue!75, thick] ({cos(deg(\alp*pi/2))*x},{sin(deg(\alp*pi/2))*x});
			\addplot [domain=0:\nlim, samples=10, blue!75, thick] ({cos(deg(\alp*pi/2))*x},{-sin(deg(\alp*pi/2))*x});
			
			\addplot [domain=0:\nlim, samples=10, red!75, thick] ({cos(deg(\bet*pi/2))*x},{sin(deg(\bet*pi/2))*x});
			\addplot [domain=0:\nlim, samples=10, red!75, thick] ({cos(deg(\bet*pi/2))*x},{-sin(deg(\bet*pi/2))*x});
			
			\addplot [domain=0:\lims, samples=250, black!75, thick] ({cos(deg(\alp*pi/2))*(abs(tan(deg(x*ln(2)/2)))^(\alp)},{x+sign(x)*sin(deg(\alp*pi/2))*(abs(tan(deg(x*ln(2)/2)))^(\alp)});
			
			\addplot [domain=0:\lims, samples=250, black!75, thick] ({cos(deg(\alp*pi/2))*(abs(tan(deg(x*ln(2)/2)))^(\alp)},{-x-sign(x)*sin(deg(\alp*pi/2))*(abs(tan(deg(x*ln(2)/2)))^(\alp)});
			
		\end{axis} 
	\end{tikzpicture}
	\caption{The corner of the domains {\color{blue} $A_\alpha$}, $\varphi_\alpha(S)$, and {\color{red}$A_{\beta}$}, for $c_0=1$, $\alpha = \frac{2}{\pi}\arctan \big( \frac{\sqrt{5}-1}{2}\big)$, and a numerical choice of $\beta$.}
	\label{fig:alphabeta}
\end{figure}

\begin{lemma}\label{lem:anglecounting} 
	There is a constant $\varrho=\varrho(c_0,\alpha) >0$ such that 
	\begin{equation}\label{eq:Nphiaasymp} 
		\mathscr{N}_{\varphi_\alpha}(w) \asymp |w|^{\frac{1}{\alpha} - 1} \dist\left(w, \partial \varphi_\alpha(\mathbb{C}_0 \cap \mathbb{D}_{\varrho})\right) 
	\end{equation}
	for every $w \in \varphi_\alpha(\mathbb{C}_0 \cap \mathbb{D}_{\varrho/2})$, where $\mathbb{D}_\varrho= \mathbb{D}(0, \rho)$. Moreover, 
	\begin{enumerate}
		\item[(i)] there is a constant $\eta_1>0$ such that if $0<\beta_1<\alpha$ is fixed, then
		\[\mathscr{N}_{\varphi_\alpha}(w) \gg (\mre{w})^{\frac{1}{\alpha}}\]
		for every $w \in A_{\beta_1} \cap \{0<\mre{w}<\eta_1\}$. 
		\item[(ii)] there are constants $\eta_2>0$ and $\alpha<\beta_2<1$ such that $\varphi_\alpha(S) \subseteq A_{\beta_2}$ and
		\[\mathscr{N}_{\varphi_\alpha}(w) \ll (\mre{w})^{\frac{1}{\alpha}}\]
		for every $w \in A_{\beta_2} \cap \{0<\mre{w}<\eta_2\}$. 
	\end{enumerate}
\end{lemma}
\begin{proof}
	Clearly, $\Phi_\alpha(2^{-s})$ is locally univalent around $s = 0$ (with $\mre s > 0$). Since
	\[\lim_{\mathbb{C}_0 \ni s \to 0} \frac{s}{\Phi_\alpha(2^{-s})} = 0,\]
	it follows from Rouch\'e's theorem that there is $\widetilde{\varrho} > 0$ such that $\varphi_\alpha$ is univalent on $\mathbb{C}_0 \cap \mathbb{D}_{\widetilde{\varrho}}$. Since $\varphi_\alpha(s) \to 0$ as $\mathbb{C}_0 \ni s \to 0$, we can also deduce that if $\delta > 0$ is sufficiently small, then for any $s \in \mathbb{C}_0 \setminus \mathbb{D}_{\widetilde{\varrho}}$,
	\[(c_0s + A_\alpha) \cap \varphi_\alpha(\mathbb{C}_0 \cap \mathbb{D}_{\delta \widetilde{\varrho}}) = \emptyset.\]
	Letting $\varrho = \delta \widetilde{\varrho}$, we therefore see that there is a uniquely defined inverse
	\[\varphi_\alpha^{-1} \colon \varphi_\alpha(\mathbb{C}_0 \cap \mathbb{D}_\varrho) \to \mathbb{C}_0 \cap \mathbb{D}_\varrho.\]
	Reparametrizing \eqref{eq:phiacalc} by letting
	\[\widetilde{t} = \operatorname{sign}(t) \left|\tan\left(\frac{\log{2}}{2}t\right)\right|^\alpha,\]
	we see that in a small neighbourhood of $w = 0$, $\partial \varphi_\alpha(\mathbb{C}_0 \cap \mathbb{D}_\varrho)$ consists of two $C^{1,\min( 1/\alpha - 1, 1)}$-smooth arcs making the angle $\alpha \pi$ at $w = 0$. Applying a Kellogg--Warschawski theorem adapted to corners \cite[Thm.~3.9]{Pommerenke92}, we thus conclude that
	\[\mathscr{N}_{\varphi_\alpha}(w) = \mre \varphi_\alpha^{-1}(w) \asymp \dist(\varphi_\alpha^{-1}(w), \mathbb{C}_0 \cap \mathbb{D}_{\varrho}) \asymp |w|^{\frac{1}{\alpha} - 1} \dist\left(w, \partial \varphi_\alpha(\mathbb{C}_0 \cap \mathbb{D}_{\varrho})\right)\]
	for every $w \in \varphi_\alpha(\mathbb{C}_0 \cap \mathbb{D}_{\varrho/2})$. Note here that we have made the restriction that $\varphi_\alpha^{-1}(w) \in \mathbb{C}_0 \cap \mathbb{D}_{\varrho/2}$ in order to ensure the validity of the first equivalence. This establishes \eqref{eq:Nphiaasymp}.
	
	We now turn to the proof of (i). From \eqref{eq:phiacalc} we have that
	\[\frac{|\mim{\varphi_\alpha(it)}|}{\mre{\varphi_\alpha(it)}} = \tan\left(\frac{\alpha \pi}{2}\right) + \frac{c_0 |t|}{\cos\left(\frac{\alpha \pi}{2}\right)\left|\tan\left(\frac{\log{2}}{2}t\right)\right|^\alpha} \geq \tan\left(\frac{\alpha \pi}{2}\right)\]
	for $-\pi/\log{2} < t < \pi/\log{2}$. Therefore, for sufficiently small $\eta_1 > 0$, we see by the maximum principle that
	\[A_\alpha \cap \{0<\mre{w}<\eta_1\} \subseteq \varphi_\alpha(\mathbb{C}_0 \cap \mathbb{D}_\varrho)\]
	and that if $w \in A_\alpha \cap \{0<\mre{w}<\eta_1\}$, then
	\[\dist\left(w, \partial\varphi_\alpha(\mathbb{C}_0 \cap \mathbb{D}_{\varrho})\right) = \dist\left(w, \partial\varphi_\alpha(\mathbb{C}_0)\right).\]
	Then, if $w \in A_{\beta_1} \cap \{0<\mre{w}<\eta_1\}$ for some fixed constant $0<\beta_1 < \alpha$, we find that
	\[\dist\left(w, \partial\varphi_\alpha(\mathbb{C}_0 \cap \mathbb{D}_{\varrho})\right) \geq \dist\left(w, \partial A_\alpha\right) \asymp \mre{w},\]
	where the implied constant depends on $\beta_1$. Combined with \eqref{eq:Nphiaasymp}, we conclude that $\mathscr{N}_{\varphi_\alpha}(w) \gg (\mre{w})^{\frac{1}{\alpha}}$ for $w \in A_{\beta_1} \cap \{0<\mre{w}<\eta_1\}$.
	
	The proof of (ii) is similar. We let $\alpha < \beta_2 < 1$ be as in Lemma~\ref{lem:anglebeta}, so that $\varphi_\alpha(S) \subseteq A_{\beta_2}$. If $\eta_2 > 0$ is sufficiently small and and if $w \in A_{\beta_2} \cap \{0<\mre{w}<\eta_2\}$, then
	\[\mathscr{N}_{\varphi_\alpha}(w) \ll |w|^{\frac{1}{\alpha} - 1} \dist\left(w, \partial A_{\beta_2} \right) \ll (\mre{w})^{\frac{1}{\alpha}}. \qedhere\]
\end{proof}
\begin{proof}
	[Proof of Theorem~\ref{thm:angle2}: Upper estimate] Let $f(s)=\sum_{m\geq1} b_m m^{-s}$ be any function in $\mathscr{H}^2$. Our goal is to establish that 
	\begin{equation}\label{eq:angle2upgoal} 
		\|\mathscr{C}_{\varphi_\alpha} f\|_{\mathscr{H}^2}^2 - |b_1|^2 \ll \sum_{m=2}^\infty \frac{|b_m|^2}{(\log{m})^{1/\alpha-1}}, 
	\end{equation}
	which will yield the stated upper bound, seeing as
	\[a_n(\mathscr{C}_{\varphi_\alpha}) = \sqrt{a_n(\mathscr{C}_{\varphi_\alpha}^\ast \mathscr{C}_{\varphi_\alpha})}.\]
	By Lemma~\ref{lem:Csum}, it is sufficient to establish \eqref{eq:angle2upgoal} for $f_j$ in $\mathscr{H}^2_j$ satisfying $f_j(+\infty)=0$, as long as we do so uniformly for every $j \in \mathbb{O}$.
	
	To estimate $\|\mathscr{C}_{\varphi_\alpha} f_j\|_{\mathscr{H}^2}^2$ we use the change of variables formula from Lemma~\ref{lem:CoV} and Lemma~\ref{lem:anglecounting} (ii). Let $\eta_2$ be as in the latter result and split the integral from Lemma~\ref{lem:CoV} at $\mre{w}=\eta_2$ to obtain
	\[I_1 = \int_{\mre{w}<\eta_2} |f_j'(w)|^2 \mathscr{N}_{\varphi_\alpha}(w)\,dw \qquad\text{and}\qquad I_2 = \int_{\eta_2 \leq \mre{w}} |f_j'(w)|^2 \mathscr{N}_{\varphi_\alpha}(w)\,dw.\]
	We begin with $I_1$. We appeal to Lemma~\ref{lem:anglecounting} (ii), then extend the integral in $\sigma$ from $(0,\eta_2)$ to $(0,\infty)$ to see that 
	\begin{equation}\label{eq:I1} 
		I_1 \ll \int_0^\infty \int_{|t|\leq \sigma \tan(\beta_2 \pi/2)} |f_j'(\sigma+it)|^2 \, dt \, \sigma^{1/\alpha}\,d\sigma 
	\end{equation}
	Expanding and using the triangle inequality, we have
	\[|f_j'(\sigma+it)|^2 \leq \sum_{m=2}^\infty \sum_{n=2}^\infty |b_m| |b_n| (\log{m})(\log{n}) (mn)^{-\sigma},\]
	where of course $b_m=0$ unless $m$ is of the form $m = j2^k$. Inserting this into \eqref{eq:I1} and computing the resulting Gamma-integral yields 
	\begin{equation}\label{eq:I1sum} 
		I_1 \ll \sum_{m=2}^\infty \sum_{n=2}^\infty \frac{|b_m| |b_n| (\log{m})(\log{n})}{(\log{m}+\log{n})^{1/\alpha+2}}. 
	\end{equation}
	By the estimate $\sqrt{xy}\leq (x+y)$ for $x,y>0$, 
	\begin{equation}\label{eq:amgm} 
		\frac{|b_m| |b_n| (\log{m})(\log{n})}{(\log{m}+\log{n})^{1/\alpha+2}} \leq \frac{|b_m|}{(\log{m})^{(1/\alpha-1)/2}}\frac{|b_n|}{(\log{n})^{(1/\alpha-1)/2}} \frac{1}{\log{m}+\log{n}}. 
	\end{equation}
	Writing $m=j 2^k \geq2$ and $n=j 2^l \geq2$, 
	\begin{equation}\label{eq:logmn} 
		\frac{1}{\log{m}+\log{n}} \leq \frac{1}{\log{2}} 
		\begin{cases}
			(k+l)^{-1} & \text{if } j=1,\\
			(1+k+l)^{-1} & \text{if } j \in \mathbb{O}\setminus\{1\}. 
		\end{cases}
	\end{equation}
	Note that if $j=1$, we are only summing over $k,l\geq1$, while we need to consider $k,l\geq0$ if $j \in \mathbb{O}\setminus\{1\}$. In either case, we insert \eqref{eq:amgm} and \eqref{eq:logmn} into \eqref{eq:I1sum} and appeal to Hilbert's inequality to conclude that
	\[I_1 \ll \sum_{m=2}^\infty \frac{|b_m|^2}{(\log{m})^{1/\alpha-1}}.\]
	The integral $I_2$ is easier to estimate. We can for instance use the coarse upper bound $\mathscr{N}_{\varphi_\alpha}(w) \leq N_{\varphi_\alpha}(w) \leq \mre{w}/c_0$ from \eqref{eq:NLest} and argue as above to see that
	\[I_2 \leq \frac{2\tan(\beta_2 \pi/2)}{c_0}\sum_{m=2}^\infty \sum_{n=2}^\infty |b_m| |b_n| (\log{m})(\log{n}) \int_{\eta_2}^\infty (mn)^{-\sigma} \sigma^2\,d\sigma.\]
	Thus $I_2$ clearly satisfies the same upper bound as $I_1$ (up to a constant), since $\eta_2>0$. 
\end{proof}
\begin{proof}
	[Proof of Theorem~\ref{thm:angle2}: Lower estimate] By Lemma~\ref{lem:Csum}, it is sufficient to establish that 
	\begin{equation}\label{eq:angle2lowergoal} 
		\|\mathscr{C}_{{\varphi_\alpha},j}\| \gg (\log{j})^{\frac{\alpha-1}{2\alpha}} 
	\end{equation}
	for $j\in\mathbb{O}\setminus\{1\}$. For such $j$ we define
	\[f_j(s) = \sum_{k=0}^\infty \frac{1}{\log(j2^k)} (j2^k)^{-s},\]
	a function in $\mathscr{H}^2_j$ which satisfies that $\|f_j\|_{\mathscr{H}^2}^2 \asymp (\log{j})^{-1}$ as $j\to\infty$. Let $\eta_1$ be as in Lemma~\ref{lem:anglecounting} (i) and let $0<\beta_1<\alpha$ be a fixed constant, to be further specified later. Using Lemma~\ref{lem:CoV} and Lemma~\ref{lem:anglecounting} (i), we find that
	\[\|\mathscr{C}_{\varphi_\alpha} f_j\|_{\mathscr{H}^2}^2 \gg \int_{A_{\beta_1} \cap \{0<\mre{w}<\eta_1\}} |f_j'(w)|^2 (\mre{w})^{1/\alpha}\,dw.\]
	Expanding $|f_j'(w)|^2$, we obtain 
	\begin{equation}\label{eq:fjexpand} 
		\|\mathscr{C}_{\varphi_\alpha} f_j\|_{\mathscr{H}^2}^2 \gg \sum_{k=0}^\infty \sum_{l=0}^\infty \int_0^{\eta_1} \int_{|t|\leq \sigma \tan(\beta_1 \pi /2)} 2^{-i(k-l)t}\,dt \, (j^2 2^{k+l})^{-\sigma} \sigma^{1/\alpha}\,d\sigma. 
	\end{equation}
	With $B=\tan(\beta_1 \pi /2)$, we have the estimate
	\[\frac{1}{2\sigma B }\int_{|t|\leq \sigma B} 2^{-i(k-l)t}\,dt = \sinc\left(\sigma B (\log{2}) (k-l)\right) \geq 1 - \frac{(\sigma B (\log{2}) (k-l))^2}{6}.\]
	We insert this estimate into the double integral in \eqref{eq:fjexpand}, then use the substitution $\sigma = x/\log(j^2 2^{k+l})$ and that $(\log{2})(k-l)/\log(j^2 2^{k+l})\leq1$ to obtain
	\[\|\mathscr{C}_{\varphi_\alpha} f_j\|_{\mathscr{H}^2}^2 \gg \frac{2B}{\big(\log(j^2 2^{k+l})\big)^{1/\alpha+2}} \int_0^{\eta_1 \log(j^2 2^{k+l})} e^{-x} x^{1/\alpha+1} \left(1 - \frac{B^2}{6}x^2\right)\,dx.\]
	By choosing $0<\beta_1<\alpha$ sufficiently small, we can make $B$ as small as we wish. In particular, we can ensure that
	\[\int_0^{\eta_1 \log{9}} e^{-x} x^{1/\alpha+1}\,dx \geq \frac{B^2}{12} \int_0^\infty e^{-x} x^{1/\alpha+3}\,dx,\]
	giving us that
	\[\|\mathscr{C}_{\varphi_\alpha} f_j\|_{\mathscr{H}^2}^2 \gg \sum_{k=0}^\infty \sum_{l=0}^\infty\frac{1}{\big(\log(j^2 2^{k+l})\big)^{1/\alpha+2}} \asymp (\log{j})^{-1/\alpha}.\]
	Hence $\|\mathscr{C}_{{\varphi_\alpha}} f_j\|_{\mathscr{H}^2}^2/\|f_j\|_{\mathscr{H}^2}^2 \asymp (\log{j})^{1-1/\alpha}$, completing the proof of \eqref{eq:angle2lowergoal}. 
\end{proof}

\bibliographystyle{amsplain-nodash} 
\bibliography{orthogonalcomp} 
\end{document}